\documentclass[12pt]{amsart}  	

\usepackage{geometry}  

\usepackage[all]{xy}						
\usepackage{amssymb}
\usepackage{amscd,latexsym,amsthm,amsfonts,amssymb,amsmath,amsxtra}
\usepackage[mathscr]{eucal}
\usepackage[colorlinks]{hyperref}
\usepackage{mathtools}

\usepackage{algorithmic}
\usepackage{fancyhdr}
\usepackage{mathrsfs}

\usepackage{verbatim}

\usepackage{marvosym}

\usepackage{chngcntr}
\numberwithin{equation}{subsection}
\newcounter{keepeqno}
\newenvironment{num}
 {\setcounter{keepeqno}{\value{equation}}%
  \begin{list}{(\theequation)}{\usecounter{equation}}%
  \setcounter{equation}{\value{keepeqno}}}
 {\end{list}}

\hypersetup{linkcolor=black, citecolor=black}

\newcommand{\BC}{{\mathbb {C}}}

\newcommand{\BN}{{\mathbb {N}}}

\newcommand{\BR}{{\mathbb {R}}}

\newcommand{\BZ}{{\mathbb {Z}}}

\newcommand{\CB}{{\mathcal {B}}}

\newcommand{\CO}{{\mathcal {O}}}

\newcommand{\CW}{{\mathcal {W}}}

\newcommand{\FA}{{\mathfrak {A}}}

\newcommand{\Fg}{{\mathfrak {g}}}

\newcommand{\Fl}{{\mathfrak {l}}}

\newcommand{\Fp}{{\mathfrak {p}}}

\newcommand{\Fz}{{\mathfrak {z}}}

\newcommand{\RM}{{\mathrm {M}}}
\newcommand{\RN}{{\mathrm {N}}}

\newcommand{\Ad}{{\mathrm{Ad}}}

\newcommand{\ab}{{\mathrm{ab}}}

\newcommand{\Gal}{{\mathrm{Gal}}}
\newcommand{\GL}{{\mathrm{GL}}}

\newcommand{\Ind}{{\mathrm{Ind}}}

\newcommand{\ord}{{\mathrm{ord}}}

\renewcommand{\Re}{{\mathrm{Re}}}
\newcommand{\reg}{{\mathrm{reg}}}
\newcommand{\Res}{{\mathrm{Res}}}

\newcommand{\SL}{{\mathrm{SL}}}

\newcommand{\Sym}{{\mathrm{Sym}}}
\newcommand{\sgn}{{\mathrm{sgn}}}

\newcommand{\tr}{{\mathrm{tr}}}

\newcommand{\vol}{{\mathrm{vol}}}

\newcommand{\wt}{\widetilde}
\newcommand{\wh}{\widehat}

\newcommand{\bs}{\backslash}

\def\alp{{\alpha}}

\def\del{{\delta}}
\def\Del{{\Delta}}
\def\diag{{\rm diag}}

\def\eps{{\epsilon}}

\def\std{\rm std}

\def\lam{{\lambda}}
\def\Lam{{\Lambda}}

\def\gam{{\gamma}}
\def\Gam{{\Gamma}}

\def\wb{\overline} 

\def\vpi{\varpi}
\def\LG{{}^{L}G}

\def\vphi{\varphi}

\def\p{\prime}
\def\LH{{}^{L}H}
\def\rss{\mathrm{rss}}

\newtheorem{thm}{Theorem}[subsection]
\newtheorem{defin}[thm]{Definition}
\newtheorem{rmk}[thm]{Remark}

\newtheorem{pro}[thm]{Proposition}
\newtheorem{lem}[thm]{Lemma}
\newtheorem{cor}[thm]{Corollary}
\newtheorem{assum}[thm]{Assumption}

\makeatletter

\newcommand{\Rmnum}[1]{\expandafter\@slowromancap\romannumeral #1@}
\makeatother

\DeclarePairedDelimiter{\ceil}{\lceil}{\rceil}

\begin{document}

\title[Stable Transfer for $\Sym^{n}$ of $\GL_{2}$]{On the Stable Transfer for $\Sym^{n}$ Lifting of $\GL_{2}$}

\author{Daniel Johnstone}
\address{School of Mathematics\\
University of Minnesota\\
Minneapolis, MN 55455, USA}
\email{johnstod@umn.edu}

\author{Zhilin Luo}
\address{School of Mathematics\\
University of Minnesota\\
Minneapolis, MN 55455, USA}
\email{luoxx537@umn.edu}



\begin{abstract}
Following the paradigm of \cite{MR3117742}, we are going to explore the stable transfer factors for $\Sym^{n}$ lifting from $\GL_{2}$ to $\GL_{n+1}$ over any local fields $F$ of characteristic zero with residue characteristic not equal to $2$. When $F=\BC$ we construct an explicit stable transfer factor for any $n$. When $n$ is odd, we provide a reduction formula, reducing the question to the construction of the stable transfer factors when the $L$-morphism is the diagonal embedding from $\GL_{2}(\BC)$ to finitely many copies of $\GL_{2}(\BC)$ under mild assumptions on the residue characteristic of $F$. With the assumptions on the residue characteristic, the reduction formula works uniformly over any local fields of characteristic zero, except that for $p$-adic situation we need to exclude the twisted Steinberg representations.
\end{abstract}

\maketitle

\tableofcontents

\section{Introduction}
Let $G$ be a connected reductive algebraic group defined over a local field $F$ of characteristic zero. 
The local reciprocity conjecture of R. Langlands \cite{langlandsproblems} asserts that the $L$-packets \cite[\S V]{MR546608} of irreducible admissible representations of $G(F)$ are parametrized by the admissible homomrphisms from the Weil-Deligne group $\CW_{F}$ to the $L$-group $\LG$ of $G$ modulo inner automorphisms of $\wh{G}$, where $\wh{G}$ is the dual group of $G$ \cite[\S III]{MR546608}. The parametrization, which is usually called the local Langlands correspondence, has been established for several cases. When $F$ is Archimedean it is known for any connected reductive algebraic groups \cite{MR1011897}. When $F$ is $p$-adic, it is known for $G=\GL_{n}$ \cite{MR1876802} \cite{MR1738446} \cite{MR3049932}, and quasi-split classical groups \cite{ar13} \cite{MR3338302}.

Langlands also proposed the functorial lifting conjecture \cite{langlandsproblems} which is a corollary of the local Langlands correspondence. Given two connected reductive algebraic groups $H$ and $G$ defined over $F$ and an $L$-homomorphism \cite[\S V]{MR546608}
$$
\rho:\LH\to \LG,
$$
there should exist a natural lifting from the $L$-packets of $H(F)$ to those of $G(F)$ along the morphism $\rho$. Moreover, the lifting should preserve temperness, i.e. the lifting of tempered $L$-packets should still be tempered $L$-packets \cite{ar13}.

From the work of Harish-Chandra \cite{MR145006} (see also \cite{SDPS99}), it is known that for any irreducible admissible representation $\pi$ of $H(F)$, the distribution character $\tr_{\pi}$ is given by a locally integrable function $\tr_{\pi}(h)$ on $H(F)$, which is locally constant (resp. analytic) on the regular semi-simple locus of $H(F)$ when $F$ is $p$-adic (resp. Archimedean). On the other hand, when $F$ is $p$-adic (resp. Archimedean), for a given a tempered $L$-packet $\Pi_{\vphi}$ of $H(F)$ associated to a tempered $L$-parameter $\vphi:\CW_{F}\to \LH$ (resp. $\vphi:W_{F}\to \LH$), it is expected that the formal sum of distribution characters
$$
\tr_{\Pi_{\vphi}} =\sum_{\pi\in \Pi_{\vphi}}
\tr_{\pi}
$$
is a stable distribution \cite{MR540901} on $H(F)$ (see also \cite[\S 8]{MR546618} and \cite{ar13}). Similarly, we have the corresponding tempered $L$-packet $\Pi_{\rho\circ \vphi}$ of $G(F)$, with an associated stable tempered character
$$
\tr_{\Pi_{\rho\circ \vphi}} = 
\sum_{\pi\in \Pi_{\rho\circ \vphi}}\tr_{\pi}
$$
on $G(F)$. In particular, the stable tempered characters $\tr_{\Pi_{\vphi}}$ and $\tr_{\Pi_{\rho\circ \vphi}}$ descend to distributions on the Steinberg-Hitchin bases $\FA_{H}$ and $\FA_{G}$ respectively \cite{MR2779866}, both of which are given by locally integrable functions on $\FA_{H}(F)$ and $\FA_{G}(F)$ that are locally constant (resp. analytic) on the regular semi-simple locus $\FA_{H}(F)^{\rss}$ and $\FA_{G}(F)^{\rss}$.

Langlands proposed a natural question in \cite{MR3117742}: Does there exist a distribution $\Theta^{\rho}(a_{H},a_{G})$ on $\FA_{H}(F)\times \FA_{G}(F)$ that yields the lifting for stable tempered characters, i.e.
\begin{equation}\label{eq:langlandsconjecture}
\tr_{\Pi_{\rho\circ \vphi}}(a_{G})
=\int_{\FA_{H}}
\Theta^{\rho}(a_{H},a_{G})\tr_{\Pi_{\vphi}}(a_{H})da_{H}.
\end{equation}

Langlands \cite{MR3117742} analyzed the question when $H(F)$ is a maximal torus in $G(F) = \SL_{2}(F)$ over any local fields of characteristic zero with residue characteristic not equal to $2$ and $\rho:\LH\to \LG$ is the natural embedding \cite[\S 3]{MR546618}. When $F$ is $p$-adic and $H(F)$ is an elliptic torus in $\SL_{2}(F)$, an elegant formula for $\Theta_{\rho}$ was obtained from the pioneer work of I. Gelfand, M. Graev and I. Pyatetski-Shapiro \cite{MR3468638}. Later it was generalized by one of the authors \cite{MR3705993} in his PhD thesis for the lifting from any unramified maximal elliptic torus $H(F)$ in $\SL_{\ell}(F)$ to $G(F)=\SL_{\ell}(F)$ under mild assumptions on the residue characteristic of $F$ and character formulas for the associated supercuspidal representations of $\SL_{\ell}$. Here $\ell$ is assumed to be a prime.

The purpose of this paper is to find an explicit construction of the stable transfer factors $\Theta^{\rho}$ for $\rho = \Sym^{n}$ lifting from $H=\GL_{2}$ to $G=\GL_{n+1}$ over any local field $F$ of characteristic zero. It is worthwhile to mention that this is the first case being investigated so far when both groups are non-abelian. We construct an explicit stable transfer factor $\Theta^{\rho}$ for $\rho= \Sym^{n}$ with all $n$ when $F=\BC$. When $n$ is odd we provide a reduction formula, reducing the construction of $\Theta^{\rho}$ to $\Theta^{\Del_{(n+1)/2}}$, which is for the case that  the $L$-morphism is the diagonal embedding $\Del_{(n+1)/2}$ from $\GL_{2}(\BC)$ to $(n+1)/2$-copies of $\GL_{2}(\BC)$. The reduction formula works uniformly over any local fields of characteristic zero under mild assumptions on the residue characteristic, except that when $F$ is $p$-adic we need to exclude the twisted Steinberg representations of $\GL_{2}(F)$.

In order to explain our results, we fix notation and conventions that will be used in the rest of this paper.

\subsubsection{Notation and Conventions.}

Throughout the paper, $F$ is always assumed to be a local field of characteristic zero with residue characteristic not equal to $2$. 
We fix a standard norm $|\cdot|$ on $F$ as follows.
\begin{itemize}
\item When $F = \BC$, let $|z| = (z\wb{z})^{1/2}$ for any $z\in \BC$ where $\wb{z}$ is the complex conjugation of $z$. Let $|z|_{\BC} = |z|^{2}$;

\item When $F=\BR$, $|\cdot|$ is the absolute value norm;

\item When $F$ is $p$-adic, let $\CO_{F}$ be the ring of integers of $F$ with maximal ideal $\Fp_{F}$. We fix a generator $\vpi$ of $\Fp_{F}$. Let $\ord:F\to \BZ\cup\{\infty \}$ be the standard valuation on $F$ satisfying $\ord(\vpi) = 1$. Let $q$ be the cardinality of the finite field $\CO_{F}/\Fp_{F}$. The norm is defined to be
$|x| = q^{-\ord(x)}$ for any $x\in F$. For any finite extension $E$ of $F$, the valuation $\ord$ and the norm $|\cdot|$ have natural extension to $E$, and the same notion will be used for $E$.
\end{itemize}

Let $\CW_{F}$ be the Weil-Deligne group of $F$ \cite{MR546607}. When $F$ is Archimedean, $\CW_{F}$ is the Weil group of $F$. When $F$ is $p$-adic, we take $\CW_{F} = W_{F}\times \SL_{2}(\BC)$ and the isomorphism between different forms of $\CW_{F}$ can be found in  \cite[\S 2.1]{MR2730575}.

For any $n\in \BN$, Let $\GL_{n}$ be the general linear group scheme defined over $\BZ$. Let the center of $\GL_{n}$ be $Z_{\GL_{n}}$.

Throughout the paper, the representation $\pi$ of $\GL_{n}(F)$ under consideration is always tempered, i.e. $\pi$ is an irreducible unitary representation of $\GL_{n}(F)$ whose matrix coefficients lie in $L^{2+\eps}(\GL_{n}(F)/Z_{\GL_{n}}(F))$ for any $\eps>0$. Let $\pi^{\infty}$ be the space of smooth vectors inside $\pi$. Sometimes we will abuse the notation to write $\pi$ for smooth vectors as well.

For any parabolic subgroup $P(F)$ of $\GL_{n}(F)$, the parabolic induction functor $\Ind^{\GL_{n}}_{P}$ is always assumed to be normalized in the sense that it sends (pre-)unitary representations to (pre-)unitary representations.

Let $B_{n}(F)$ be the standard upper triangular Borel subgroup of $\GL_{n}(F)$ with Levi subgroup given by the diagonal split torus $T_{n}(F)$.

Define the standard upper triangular parabolic subgroup $P_{n}(F)$ of $\GL_{n}(F)$ with Levi decomposition $P_{n}(F)= M_{n}(F)N_{n}(F)$ as follows,
\begin{itemize}
\item when $n$ is even, $P_{n}(F)$ is the standard parabolic subgroup associated to the partition $(2,2,...,2)$. Then $M_{n}(F)\simeq \GL_{2}(F)\times...\times \GL_{2}(F)$;

\item when $n$ is odd, $P_{n}(F)$ is the standard parabolic subgroup associated to the partition $(2,2,...,2,1)$. Then $M_{n}(F)\simeq \GL_{2}(F)\times ...\times \GL_{2}(F)\times \GL_{1}(F)$.
\end{itemize}

Let $\GL_{n}(F)^{\rss}$ be the set of regular semi-simple elements in $\GL_{n}(F)$. 

For any $x\in \GL_{n}(F)^{\rss}$, let $D_{\GL_{n}}(x) = |\det(1-\Ad(x))_{\Fg\Fl_{n}(F)/(\Fg\Fl_{n}(F))_{x}}|$ be the Weyl discriminant of $x$, where $\Fg\Fl_{n}(F)$ is the Lie algebra of $\GL_{n}(F)$, and $(\Fg\Fl_{n}(F))_{x}$ is the centralizer of $x$ in $\Fg\Fl_{n}(F)$ under the adjoint action $\Ad$ of $\GL_{n}(F)$ on $\Fg\Fl_{n}(F)$.

\subsubsection{Main results.}

The goal of the present paper is to explore the stable transfer factors when $H=\GL_{2}$, $G=\GL_{n+1}$, and $\rho$ is the $\Sym^{n}$ lifting from $\LH\simeq \GL_{2}(\BC)$ to $\LG\simeq \GL_{n+1}(\BC)$. 

Before introducing our results, let us mention that for $\GL_{n}$, by Hilbert 90 \cite{serregaloiscoh} and the cohomological parametrization of stable conjugacy classes \cite{MR540901}, the concept of stable conjugacy is equivalent to ordinary conjugacy. Therefore the $L$-packets of $\GL_{n}$ are always singleton \cite[\S 1]{MR546618}. It follows that for any (tempered) representation $\pi$ of $\GL_{n}(F)$, the stable character of $\pi$ is just the usual distribution character of $\pi$.

In Section \ref{sec:functorialimage}, we establish the following theorem, which says that the functorial image of tempered $L$-parameters of $\GL_{2}(F)$ are \emph{almost} parallel over any local fields of characteristic zero with residue characteristic not equal to $2$.

\begin{thm}\label{thm:functorialim:parallel}
For any tempered $L$-parameter $\vphi$ of $\GL_{2}(F)$, the following description for $\Sym^{n}\circ \vphi$ holds.
\begin{itemize}
\item When $F = \BC$, any tempered $L$-parameter $\vphi:W_{F}\to \GL_{2}(\BC)$ is of the form $\vphi = \chi_{1}\oplus \chi_{2}$, where $\chi_{i}$ is a unitary character of $F^{\times}$ $(i=1,2)$. Then
$$
\Sym^{n}\circ \vphi = 
\bigg\{
\begin{matrix}
\bigoplus_{k=1}^{n/2}
\big(
(\chi_{1}\chi_{2})^{(n-2k)/2}\otimes (\chi_{1}^{2k}\oplus \chi_{2}^{2k})\big) \bigoplus (\chi_{1}\chi_{2})^{n/2}, & \text{ if $n$ is even},\\
\bigoplus_{k=1}^{(n+1)/2}
\big(
(\chi_{1}\chi_{2})^{(n-2k+1)/2}\otimes (\chi^{2k-1}_{1}\oplus \chi^{2k-1}_{2})
\big)
, & \text{ if $n$ is odd}.
\end{matrix}
$$
The same description holds for real (resp. $p$-adic) field $F$ and $\vphi$ is a reducible tempered $L$-parameter of $W_{F}$ (resp. $\CW_{F}$).

\item When $F =\BR$, any irreducible tempered $L$-parameter $\vphi:W_{F}\to \GL_{2}(\BC)$ is parametrized by a pair $(l,t)\in \BN\times i\BR$ which is denoted by $\vphi_{(l,t)}$. Then
$$
\Sym^{n}\circ \vphi_{(l,t)} = 
\bigg\{
\begin{matrix}
\bigoplus_{k=1}^{n/2}
\big(
\chi_{(-1)^{l(n/2-k)},(n-2k)t} \otimes \vphi_{(2kl,2kt)}
\big) \bigoplus \chi_{(-1)^{nl/2}, nt}, & \text{ if $n$ is even},\\
\bigoplus_{k=1}^{(n+1)/2}
\big(
\chi_{(-1)^{(n-2k+1)l/2}, (n-2k+1)t} \otimes \vphi_{(2k-1)l, (2k-1)t}
\big)
, & \text{ if $n$ is odd}.
\end{matrix}
$$
Here for any $(i,t)\in \{\pm 1 \}\times \BC$ the quasi-character $\chi_{i,t}$ of $W_{F} = \BC^{\times}\cup j\BC^{\times}$ is defined as follows,
$$
\chi_{i,t}(z) = |z|^{t}, z\in \BC^{\times},\quad  \chi_{i,t}(j) = i.
$$

\item When $F$ is $p$-adic with residue characteristic not equal to $2$, any irreducible tempered $L$-parameter $\vphi:\CW_{F}\to \GL_{2}(\BC)$ is
\begin{enumerate}
\item\label{twistedsteinberg}
the tensor product of a unitary character $\chi$ of $W_{F}\to W_{F}^{\ab}\simeq F^{\times}$ \cite{MR0220701} with the standard representation $\std$ of $\SL_{2}(\BC)$. Then we have
$$
\Sym^{n}\circ (\chi\otimes \mathrm{std}) = \chi^{n}\otimes \Sym^{n};
$$

\item 
the representation $\Ind^{W_{F}}_{W_{E}}\theta$ obtained from an admissible pair $(E/F,\theta)$ \cite[18.2]{gl2llc}, where $E/F$ is a quadratic extension, $\theta$ is an admissible character of $W_{E}\to W^{\mathrm{ab}}_{E}\simeq E^{\times}$ \cite{MR0220701}, and $\theta|_{F^{\times}}$ is unitary.
Then
$$
\Sym^{n}\circ \Ind^{W_{F}}_{W_{E}}\theta = 
\bigg\{
\begin{matrix}
\bigoplus_{k=1}^{n/2}
\big(
\theta^{(n-2k)/2}\otimes \Ind^{W_{F}}_{W_{E}}\theta^{2k}
\big)
\bigoplus (\theta)^{n/2}, & \text{ if $n$ is even},\\
\bigoplus_{k=1}^{(n+1)/2}
\big(
\theta^{(n-2k+1)/2}\otimes\Ind^{W_{F}}_{W_{E}}\theta^{2k-1}
\big), & \text{ if $n$ is odd}.
\end{matrix}
$$
Here we abuse the notation viewing $\theta$ as a quasi-character of $W_{F}\to W^{\mathrm{ab}}_{F}\simeq F^{\times}$ via restriction.
\end{enumerate}
\end{itemize}
\end{thm}
From Theorem \ref{thm:functorialim:parallel}, we observe that except situation (\ref{twistedsteinberg}) in the case when $F$ is $p$-adic, the description of the functorial lifting of tempered $L$-parameters of $\GL_{2}(F)$ shares strong similarities with each other over any local field $F$ of characteristic zero with residue characteristic not equal to $2$. We will explicate more details in Section \ref{sec:functorialimage}.

Based on \cite[\S 3]{MR0499005} \cite[\S 3]{MR2684298}, we find that except situation (\ref{twistedsteinberg}) in the $p$-adic case, the tempered representation of $\GL_{n+1}(F)$ associated to $\Sym^{n}\circ \vphi$ is always parabolic induced from $P_{n+1}(F)$. More precisely the following theorem holds.

\begin{thm}\label{thm:lifttolevi}
Following the notation of Theorem \ref{thm:functorialim:parallel}, for any tempered $L$-parameter $\vphi$ of $\GL_{2}(F)$, except situation (\ref{twistedsteinberg}) in the $p$-adic case, the tempered representation associated to the tempered $L$-parameter $\Sym^{n}\circ \vphi$ is always of the form $\Ind^{\GL_{n+1}}_{P_{n+1}}\pi_{M_{n+1}}\otimes 1_{N_{n+1}}$. Here the tempered representation $\pi_{M_{n+1}}$ of $M_{n+1}(F)$ has the following description.
\begin{itemize}
\item When $F=\BC$ and $\vphi = \chi_{1}\oplus \chi_{2}$, 
$$
\pi_{M_{n+1}}= 
\bigg\{
\begin{matrix}
\big\{
\bigotimes_{k=1}^{n/2}
\big(
(\chi_{1}\chi_{2})^{(n-2k)/2}
\otimes \Ind_{B_{2}}^{\GL_{2}}(\chi^{2k}_{1},\chi^{2k}_{2})
\big) \big\}\bigotimes (\chi_{1}\chi_{2})^{n/2}, & \text{ if $n$ is even},\\
\bigotimes_{k=1}^{(n+1)/2}\big(
(\chi_{1}\chi_{2})^{(n-2k+1)/2}
\otimes \Ind_{B_{2}}^{\GL_{2}}(\chi^{2k-1}_{1},\chi^{2k-1}_{2})
\big), & \text{ if $n$ is odd}.
\end{matrix}
$$
The same descriptions holds for any tempered principal series of $\GL_{2}$ over real and $p$-adic fields.

\item When $F=\BR$ and $\vphi = \vphi_{(l,t)}$,
$$
\pi_{M_{n+1}} = 
\bigg\{
\begin{matrix}
\{
\bigotimes_{k=1}^{n/2}
((-1)^{l(n/2-k)}, (n-2k)t)\otimes 
\pi_{(2kl, 2kt)}
\}
\bigotimes ((-1)^{nl/2},nt),  & \text{ if $n$ is even},
\\
\bigotimes_{k=1}^{(n+1)/2}
((-1)^{l(n-2k+1)/2}, (n-2k+1)t/2)\otimes 
\pi_{((2k-1)l, (2k-1)t)},  & \text{ if $n$ is odd}.\\
\end{matrix}
$$ 
Here for any $(l,t)\in \BN\times \BC$, $\pi_{(l,t)}$ is the discrete series representation $D_{l}\otimes |\det(\cdot)|^{t}$ of $\GL_{2}(F)$ \cite[\S 2]{knapp55local}. For any $(i,t)\in \{\pm \}\times \BC$, we abuse the notation, denoting $(i,t)$ for the character of $\BR^{\times}$ given by
\begin{align*}
\begin{matrix}
a\in \BR^{\times}\to |a|^{t},& \text{ if $i = 1$},\\
a\in \BR^{\times}\to \sgn(a)|a|^{t}, & \text{ if $i=-1$},
\end{matrix}
\end{align*}
where $\sgn$ is the sign character of $\BR^{\times}$.

\item When $F$ is $p$-adic with residue characteristic not equal to $2$ and $\vphi = \Ind^{W_{F}}_{W_{E}}\theta$,
$$
\pi_{M_{n+1}} = \bigg\{
\begin{matrix}
\bigotimes_{k=1}^{n/2}
\big(
\theta^{(n-2k)/2}\otimes \pi_{\theta^{2k}}
\big)
\bigotimes (\theta)^{n/2}, &\text{ if $n$ is even},\\
\bigotimes_{k=1}^{(n+1)/2}
\theta^{(n-2k+1)/2}\otimes \pi_{\theta^{2k-1}}
, & \text{ if $n$ is odd}.
\end{matrix}
$$
Here for any $k\in \BN$, $\pi_{\theta^{k}}$ is the tempered representation of $\GL_{2}(F)$ with tempered $L$-parameter $\Ind^{W_{F}}_{W_{E}}\theta^{k}$.
\end{itemize}
\end{thm}

In particular, we will see in Section \ref{sec:stabletransfer} that for a tempered representation $\pi$ of $\GL_{2}(F)$,
\begin{itemize}
\item
when $\pi$ is a tempered principal series representation $\Ind_{B_{2}}^{\GL_{2}}(\chi_{1},\chi_{2})$, the unitary character $\chi_{1}\chi_{2}$ is the central character of $\pi$; 

\item
when $F$ is real and $\pi = \pi_{(l,t)}$ is the discrete series representation $D_{l}\otimes |\det(\cdot)|^{t}$, the unitary character $((-1)^{l},2t)$ is the central character of $\pi$;

\item
when $F$ is $p$-adic with residue characteristic not equal to $2$ and $\pi =\pi_{\theta}$ is a supercuspidal representation, $\theta|_{F^{\times}}$ is the central character of $\pi$.
\end{itemize}

Therefore at least from their formal expressions, we observe that the lifting of tempered representations of $\GL_{2}(F)$ shares strong similarities over any local fields of characteristic zero with residue characteristic not equal to $2$. 

In general, from the character formula of induced representations \cite{Gv72} \cite[(10.27)]{MR855239}, we only need to determine a distribution $\Theta^{M_{n+1}}$ that gives rise to the lifting of distribution characters from $\tr_{\pi}$ to $\tr_{\pi_{M_{n+1}}}$. When $n$ is odd, based on mild assumptions on the residue characteristic of $F$, from the following theorem we are able to reduce the construction of $\Theta^{M_{n+1}}$ to $\Theta^{\Del_{(n+1)/2}}$, where 
$\Del_{(n+1)/2}$ is the $L$-morphism given by the diagonal embedding from $\GL_{2}(\BC)$ to $(\GL_{2}(\BC))^{(n+1)/2}\simeq {}^LM_{n+1}$. More precisely, the following theorem is established in Section \ref{sec:stabletransfer}.

\begin{thm}\label{thm:characterrelation}
The following identity for distribution characters holds.
\begin{itemize}
\item For any local field $F$ of characteristic zero, fix a tempered principal series representation $\Ind^{\GL_{2}}_{B_{2}}(\chi_{1},\chi_{2})$ of $\GL_{2}(F)$. Then for any $k\in \BN$, the following equality of locally integrable distributions holds.
$$
D_{\GL_{2}}^{1/2}(\gam^{k})\tr_{\Ind(\chi_{1},\chi_{2})}(\gam^{k}) = D_{\GL_{2}}^{1/2}(\gam)\tr_{\Ind(\chi^{k}_{1},\chi^{k}_{2})}(\gam), \quad \gam\in \GL_{2}(F)^{\rss}.
$$

\item When $F=\BR$, let $\pi_{(l,t)}$ be the discrete series representation of $\GL_{2}(F)$ associated to $(l,t)\in \BN\times \BC$. Then the following equality of locally integrable distributions holds whenever $k\in \BN$ is odd.
$$
D_{\GL_{2}}^{1/2}(\gam^{k}) \tr_{\pi_{(l,t)}}(\gam^{k}) = D_{\GL_{2}}^{1/2}(\gam) \tr_{\pi_{(kl,kt)}}(\gam),\quad \gam\in \GL_{2}(F)^{\rss}.
$$

\item When $F$ is $p$-adic, let $\pi_{\theta}$ be the supercuspidal representation of $\GL_{2}(F)$ associated to an admissible pair $(E/F,\theta)$. Then the following equality of locally integrable distributions holds whenever $k\in \BN$ is coprime to $(q-1)q(q+1)$.
$$
D^{1/2}_{\GL_{2}}(\gam^{k})\tr_{\pi_{\theta}}(\gam^{k}) = D^{1/2}_{\GL_{2}}(\gam)\tr_{\pi_{\theta^{k}}}(\gam),\quad \gam\in \GL_{2}(F)^{\rss}.
$$
\end{itemize}
\end{thm}
\begin{rmk}\label{rmk:partialresult}
When $k$ does not satisfy the assumptions stated in Theorem \ref{thm:characterrelation}, we also obtain similar results under mild assumptions on the residue characteristic of $F$. But the formula is not as uniform as the tempered principal series situation. In particular we cannot unify the tempered principal series and discrete series. For details the readers may consult Section \ref{sec:stabletransfer}.
\end{rmk}

As a corollary, under the assumptions made in Theorem \ref{thm:characterrelation}, we are able to reduce the construction of the distribution $\Theta^{\rho}$ in (\ref{eq:langlandsconjecture}) for $\rho =\Sym^{n}$ to $\Theta^{\Del_{(n+1)/2}}$ where $\Del_{(n+1)/2}$ is the $L$-morphism given by the diagonal embedding from $\GL_{2}(\BC)$ to ${}^LM_{n+1}$. More precisely, we first make the following assumption on the existence of the distribution $\Theta^{\Del_{(n+1)/2}}$ which gives the identity (\ref{eq:langlandsconjecture}) for $\Del_{(n+1)/2}$ when $n$ is odd.
\begin{assum}\label{assum}
There exists an invariant distribution on $M_{n+1}(F)\times \GL_{2}(F)$, which is denoted by $\Theta^{\Del_{(n+1)/2}}$, such that the following distributional identity holds
$$
D^{1/2}_{M_{n+1}}(m)\tr_{\bigotimes_{k=1}^{(n+1)/2}\pi}(m) = 
\int_{\GL_{2}(F)}
\Theta^{\Del_{(n+1)/2}}(m,\gam)
D^{1/2}_{\GL_{2}}(\gam)
\tr_{\pi}(\gam)
$$
for any $(m,\gam)\in M_{n+1}(F)^{\rss}\times \GL_{2}(F)^{\rss}$.
\end{assum}
We make a remark that the functoriality associated to the $L$-morphism $\Del_{(n+1)/2}$ is given by the diagonal lifting from $\pi\in \mathrm{Irr}(\GL_{2}(F))$ to $\bigotimes_{k=1}^{(n+1)/2}\pi\in \mathrm{Irr}(M_{n+1}(F))$, which is quite straightforward. However, proving the existence of an invariant distribution $\Theta^{\Del_{(n+1)/2}}$ satisfying (\ref{eq:langlandsconjecture}) seems to be a hard question from analytical point of view. At his point, we decide to take is as an assumption and hope to get back to it in future work.

By straighforward computation, we obtain the following corollary.

\begin{cor}\label{cor:reducetoDel}
Based on the assumptions made in Theorem \ref{thm:characterrelation} and Assumption \ref{assum}, define the following distribution $\Theta^{M_{n+1}}$ on $M_{n+1}(F)\times \GL_{2}(F)$,
\begin{align*}
&\Theta^{M_{n+1}}
(g_{1},...,g_{(n+1)/2},\gam) 
\\
&= 
\Theta^{\Del_{(n+1)/2}}
(\det(g_{1})^{(n-1)/2}\cdot g_{1},\det(g_{2})^{(n-3)/2}\cdot g_{2},..., \det(g_{(n+1)/2})\cdot g_{(n-1)/2}, g_{(n+1)/2},\gam),
\end{align*}
where $(g_{1},...,g_{(n+1)/2})\in \prod_{k=1}^{(n+1)/2}\GL_{2}(F)\simeq M_{n+1}(F)$ and $\gam\in \GL_{2}(F)$.

Then the following distributional identity holds,
$$
D^{1/2}_{M_{n+1}}(m)
\tr_{\pi_{M_{n+1}}}(m)
=\int_{\GL_{2}(F)}
\Theta^{M_{n+1}}(m,\gam)D^{1/2}_{\GL_{2}}(\gam)\tr_{\pi}(\gam)
$$
for 
\begin{itemize}
\item any tempered representations of $\GL_{2}(F)$ when $F$ is Archimedean;

\item any tempered representations except twisted Steinberg representations of $\GL_{2}(F)$ when $F$ is $p$-adic.
\end{itemize}
and any $(m,\gam)\in M_{n+1}(F)^{\rss}\times \GL_{2}(F)^{\rss}$.
\end{cor}

From $\Theta^{M_{n+1}}$ we are able to write down the distribution $\Theta^{\rho}$ on $\GL_{n+1}(F)\times \GL_{2}(F)$ that gives rise to the identity (\ref{eq:langlandsconjecture}) where $\rho = \Sym^{n}$. This is done by using the character formula of induced representations \cite[Theorem~3]{Gv72} \cite[10.21,10.26]{MR855239}. We can also view $\Theta^{\rho}$ as the parabolic induction of the invariant distribution $\Theta^{M_{n+1}}$ from $M_{n+1}(F)\times \GL_{2}(F)$ to $\GL_{n+1}(F)\times \GL_{2}(F)$ with explicit formula given in \cite[(3.4.2)]{beuzart2015local}. Here we note that for the Archimedean case, the character formula in \cite[10.26]{MR855239} is not exactly the same as \cite[Theorem~3]{Gv72}. But the identity \cite[10.21]{MR855239} and the argument in \cite[p.237]{Gv72} imply that the character formula \cite[Theorem~3]{Gv72} also holds in the Archimedean case.

\begin{cor}\label{cor:distributiongroup}
Based on the assumptions made in Theorem \ref{thm:characterrelation} and Assumption \ref{assum}, define the distribution $\Theta^{\rho}$ on $\GL_{n+1}(F)\times \GL_{2}(F)$ as follows. It is supported on the conjugacy classes of $\GL_{n+1}(F)\times \GL_{2}(F)$-regular semi-simple elements in $M_{n+1}(F)\times \GL_{2}(F)$. For any $m\in M_{n+1}(F)^{\GL_{n+1}(F)\text{-}\reg}$ whose centralizer in $\GL_{n+1}(F)$ is a maximal torus denoted as $A_{m}$, let
$$
\Theta^{\rho}(m,\gam) = 
\sum_{s\in W(A_{M_{n+1}},A_{m})}
\Theta^{M_{n+1}}(sm,\gam).
$$
Here $A_{M_{n+1}}$ is the split center of $M_{n+1}(F)$ and the set $W(A_{M_{n+1}},A_{m})$ is defined in \cite[p.237]{Gv72}. 

Then the following distributional identity holds for any $(g,\gam)\in \GL_{n+1}(F)^{\rss}\times \GL_{2}(F)^{\rss}$,
$$
D^{1/2}_{\GL_{n+1}}(g)
\tr_{\Ind_{P_{n+1}}^{\GL_{n+1}}\pi_{M_{n+1}}\otimes 1_{N_{n+1}}}(g)
=\int_{\GL_{2}(F)}
\Theta^{\rho}(g,\gam)D^{1/2}_{\GL_{2}}(\gam)\tr_{\pi}(\gam)
$$
for 
\begin{itemize}
\item any tempered representations of $\GL_{2}(F)$ when $F$ is Archimedean;

\item any tempered representations except twisted Steinberg representations of $\GL_{2}(F)$ when $F$ is $p$-adic.
\end{itemize}
\end{cor}

Finally let us note that when $F=\BC$ there is a much more simplified expression for the distribution $\Theta^{\rho}$ with $\rho = \Sym^{n}$ that works uniformly for any $n$. Note that when $F=\BC$ any tempered representations of $\GL_{2}(F)$ are tempered principal series, and their functorial lifting are also tempered principal series of $\GL_{n+1}(F)$. In particular their distribution characters are determined by their induced datum on the maximal split torus. Therefore the following result is not hard to derive.
\begin{thm}\label{thm:distributioncomplex}
When $F=\BC$, define the distribution $\Theta^{\rho}$ on $\GL_{n+1}(F)\times \GL_{2}(F)$ as follows. It is invariant under the conjugation action of $\GL_{n+1}(F)\times \GL_{2}(F)$, and its value on the diagonal maximal split torus is given by
\begin{equation}\label{eq:distributioncomplex}
\Theta^{\rho}(T,t) = 
\frac{1}{2}
\sum_{w\in S_{n+1}}\del_{\prod_{k=1}^{n+1}T^{n+1-k}_{w(k)}}(t_{1})\otimes \del_{\prod_{k=1}^{n+1}T^{k-1}_{w(k)}}(t_{2}),
\end{equation}
where $T = (T_{1},...,T_{n+1})$ (resp. $t = (t_{1},t_{2})$) is a diagonal regular semi-simple element in $\GL_{n+1}(F)$ (resp. $\GL_{2}(F)$), $\del_{*}(*)$ is the delta distribution, and $S_{n+1}$ is the permutation group in $(n+1)$ letters $\{1,...,n+1 \}$.

Then the following distributional identity holds
$$
D^{1/2}_{\GL_{n+1}}(g) 
\tr_{\Ind_{B_{n+1}}^{\GL_{n+1}}(\chi^{n}_{1},\chi^{n-1}_{1}\chi_{2},...,\chi^{n}_{2})}
(g) 
=\int_{\GL_{2}(F)}
\Theta^{\rho}(g,\gam)
D^{1/2}_{\GL_{2}}(\gam)\tr_{\Ind^{\GL_{2}}_{B_{2}}(\chi_{1},\chi_{2})}(\gam)
$$
for any $(g,\gam)\in \GL_{n+1}(F)^{\rss}\times \GL_{2}(F)^{\rss}$.
\end{thm}
Note that actually the distribution (\ref{eq:distributioncomplex}) works for tempered principal series of $\GL_{2}(F)$ over any local field $F$ of characteristic zero.

\paragraph{\textbf{Acknowledgement.}}
Z. Luo would like to thank his advisor D. Jiang for reading the manuscript carefully with helpful feedback, and continued encouragement. The work of Z. Luo is supported in part through NSF Grant: DMS-1901802 of D. Jiang.

\section{Functorial image}\label{sec:functorialimage}

In this section, following \cite{knapp55local} and \cite[\S 8]{gl2llc}, we are going to review the local Langlands correspondence for $\GL_{2}(F)$ where $F$ is any local field of characteristic zero with residue characteristic not equal to $2$. We also calculate the functorial images of tempered Langlands parameters of $\GL_{2}(F)$ along $\rho = \Sym^{n}$. We are going to see that the functorial images of tempered Langlands parameters share strong similarities over any local field $F$ of characteristic zero with residue characteristic not equal to $2$. In particular we establish Theorem \ref{thm:functorialim:parallel} and Theorem \ref{thm:lifttolevi} in this section.

\subsection{Archimedean case}

\subsubsection{Complex case}
Let $F=\BC$. The Weil group $W_{F}$ of $F$ is isomorphic to $\BC^{\times}$. For any quasi-character $\chi$ of $\BC^{\times}$, there exists $l\in \BZ$ and $t\in \BC$ such that 
$$
\chi(z)= [z]^{l}|z|^{t}_{\BC}, \quad z\in \BC^{\times}
$$
where $[z]= z/|z|$.

The set of equivalence classes of irreducible admissible representations of $\GL_{2}(F)$ is parametrized by the equivalence classes of $2$-dimensional semi-simple complex representations of $W_{F}$. More precisely, for any $2$-dimensional semi-simple complex representation $\vphi$ of $W_{F}$, up to equivalence $\vphi$ is the direct sum of two quasi-characters $\chi_{1}$ and $\chi_{2}$, with parameters $(l_{1},t_{1})$ and $(l_{2},t_{2})$ in $\BZ\times \BC^{\times}$ satisfying $\Re(t_{1})\geq \Re(t_{2})$. We form the induced representation $\Ind^{\GL_{2}}_{B_{2}}(\chi_{1},\chi_{2})$. The representation has a unique irreducible quotient which we denote by $J(\chi_{1},\chi_{2})$. The construction exhausts all the irreducible admissible representations of $\GL_{2}(F)$, and the correspondence $\vphi = \chi_{1}\oplus \chi_{2}\leftrightarrow J(\chi_{1},\chi_{2})$ gives us the desired bijective parametrization.

From \cite[\S 3]{MR2684298} the tempered representations of $\GL_{2}(F)$ are parametrized by the unitary characters $\chi_{1},\chi_{2}$ of $F^{\times}$, and they are given by the fully induced principal series representations $\Ind^{\GL_{2}}_{B_{2}}(\chi_{1},\chi_{2})$. Similar results hold for $\GL_{n}(F)$ as well. It follows that we can draw the following conclusion about the functorial image of $\Ind^{\GL_{2}}_{B_{2}}(\chi_{1},\chi_{2})$ along $\rho = \Sym^{n}$.

\begin{lem}\label{functorial:complex}
Given a tempered principal series representation $\Ind^{\GL_{2}}_{B_{2}}(\chi_{1},\chi_{2})$ of $\GL_{2}(F)$, its functorial image along $\rho = \Sym^{n}$ is given by the following fully induced tempered principal series representation of $\GL_{n+1}(F)$
$$
\Ind^{\GL_{n+1}}_{B_{n+1}}(\chi^{n}_{1},\chi^{n-1}_{1}\chi_{2},...,\chi^{n}_{2}).
$$
\end{lem}

\begin{rmk}\label{rmk:complex}
For later purpose, using parabolic induction by stages, we will also write the functorial image as follows. We recall the standard parabolic subgroup $P_{n+1}(F)=M_{n+1}(F)N_{n+1}(F)$ introduced in the introduction. Consider the following representation on the Levi factor $M_{n+1}(F)$
$$
\pi_{M_{n+1}}= 
\bigg\{
\begin{matrix}
\big\{
\bigotimes_{k=1}^{n/2}
\big(
(\chi_{1}\chi_{2})^{(n-2k)/2}
\otimes \Ind_{B_{2}}^{\GL_{2}}(\chi^{2k}_{1},\chi^{2k}_{2})
\big) \big\}\bigotimes (\chi_{1}\chi_{2})^{n/2}, & \text{ if $n$ is even},\\
\bigotimes_{k=1}^{(n+1)/2}
\big(
(\chi_{1}\chi_{2})^{(n-2k+1)/2}
\otimes \Ind_{B_{2}}^{\GL_{2}}(\chi^{2k-1}_{1},\chi^{2k-1}_{2})
\big) , & \text{ if $n$ is odd}.
\end{matrix}
$$
Then
$$
\Ind_{B_{n+1}}^{\GL_{n+1}}
(\chi_{1}^{n},\chi^{n-1}_{1}\chi_{2},..., \chi^{n}_{2})
\simeq 
\Ind^{\GL_{n+1}}_{P_{n+1}}\pi_{M_{n+1}}\otimes 1_{N_{n+1}}.
$$
\end{rmk}

\begin{rmk}\label{rmk:principalseries}
Following the explicit description of tempered representations of $\GL_{n}(F)$ over any local field $F$ of characteristic zero \cite[\S 3]{MR0499005} \cite[\S 3]{MR2684298}, we immediately realize that the same description in Remark \ref{rmk:complex} holds for any tempered principal series representations of $\GL_{2}(F)$ over any local field $F$ of characteristic zero. 
\end{rmk}

\subsubsection{Real case}
Let $F=\BR$. The Weil group $W_{F}$ can be identified with the following non-split extension of $\BC^{\times}$,
$$
\BC^{\times}\cup j\BC^{\times},
$$
where $j^{2} = -1$ and $jzj^{-1} = \wb{z}$ for any $z\in \BC^{\times}$.

Any finite-dimensional semi-simple representations of $W_{F}$ are completely reducible, and the irreducible ones are either $1$ or $2$ dimensional. In the following we are going to have a quick review for the irreducible semi-simple representations of $W_{F}$. 
\begin{itemize}
\item The set of equivalence classes of $1$-dimensional representations of $W_{F}$ is parametrized by the set $\{\pm 1 \}\times \BC$. More precisely, for any $(i,t)\in \{\pm 1 \}\times \BC$, 
it is associated to the following quasi-character $\chi$ of $W_{F}$,
$$
\chi:z\in \BC^{\times}\to |z|^{t}, \quad \chi(j) = i.
$$
For convenience we will denote $i$ by the sign $\pm$ and use the same notation $(i,t)$ to denote the associated quasi-character.

\item The set of equivalence classes of $2$-dimensional irreducible semi-simple representations of $W_{F}$ is parametrized by the set $\BN\times \BC$. They admit the following description \cite[3.3]{knapp55local}. For any $\vphi = \vphi_{(l,t)}$ an irreducible $2$-dimensional representation of $W_{F}$ parametrized by $(l,t)\in \BN\times \BC$, there exists an ordered basis $\{v_{1},v_{2} \}$ of $\vphi$ such that we have the following matrix presentation 
$$
\vphi|_{\BC^{\times}}(z) = 
\left(
\begin{matrix}
|z|^{t}(\frac{z}{|z|})^{l} & \\
& |z|^{t}(\frac{z}{|z|})^{-l}
\end{matrix}
\right),
\quad 
\vphi(j) = 
\left(
\begin{matrix}
	& (-1)^{l}\\
1 	&
\end{matrix}
\right).
$$
\end{itemize}

\begin{rmk}
We point out a small typo appearing in \cite[3.3]{knapp55local}.
In \cite[3.3]{knapp55local}, when $\vphi$ is associated to $(l,t)\in \BN\times \BC$, the restriction of $\vphi$ to $\BC^{\times}$ has the following matrix presentation under the ordered basis $\{ v_{1},v_{2}\}$ described above
$$
\vphi|_{\BC^{\times}}(z) =
\left(
\begin{matrix}
|z|^{2t}(\frac{z}{|z|})^{l} & \\
& |z|^{2t}(\frac{z}{|z|})^{-l}
\end{matrix}
\right) = |z|^{2t}
\left(
\begin{matrix}
(\frac{z}{|z|})^{l} & \\
& (\frac{z}{|z|})^{-l}
\end{matrix}
\right).
$$
But if we use this description, then $\vphi_{(l,t)} \simeq \vphi_{(l,0)}\otimes (0,2t)$ where $(0,2t)$ is the quasi-character of $W_{F}$ sending $z\in \BC^{\times}$ to $|z|^{2t}$ and $j$ to $1$. It turns out that the irreducible admissible representation of $\GL_{2}(\BR)$ associated to $\vphi_{(l,0)}\otimes (0,2t)$ should be $D_{l}\otimes |\det (\cdot)|^{2t}$ if we follow the usual compatibility of local Langlands correspondence for $\GL_{2}(\BR)$ with unramified twist, which is not $D_{l}\otimes |\det(\cdot)|^{t}$ as claimed in \cite[3.3]{knapp55local}. Therefore we arrive at a contradiction. For a more precise reference, we refer to \cite[\S 1, Proposition~2]{prasad2017reducible} and \cite[\S 3.1]{MR3748235}, where in the latter reference, the author switches $t$ to $2t$ uniformly for the $L$-parameters of quasi-characters, which is actually equivalent to our notation.
\end{rmk}

The set of equivalence classes of irreducible admissible representations of $\GL_{2}(F)$ is parametrized by the set of  equivalence classes of $2$-dimensional semi-simple complex representations of $W_{F}$. More precisely, for any $2$-dimensional semi-simple complex representation $\vphi$ of $W_{F}$, the following description for the associated irreducible admissible representations of $\GL_{2}(F)$ holds.

\begin{itemize}
\item When $\vphi$ is reducible, $\vphi$ is the direct sum of two quasi-characters $\chi_{1}$ and $\chi_{2}$ of $W_{F}$ with parameters $(i_{1},t_{1}),(i_{2},t_{2})\in \{\pm 1 \}\times \BC$. For a quasi-character $\chi$ of $W_{F}$ with parameter $(i,t)\in \{\pm 1 \}\times \BC$, the following description holds.
\begin{enumerate}
\item When $i = +$, $\chi$ is defined as follows
$$
a\in \GL_{1}(F)\to |a|^{t};
$$

\item When $i = -$, $\chi$ is defined as follows
$$
a\in \GL_{1}(F)\to \sgn(a)|a|^{t};
$$
\end{enumerate}
For $\vphi = \chi_{1}\oplus \chi_{2}$ with $\Re(t_{1})\geq \Re(t_{2})$. Consider the induced representation $\Ind^{\GL_{2}}_{B_{2}}(\chi_{1},\chi_{2})$. Parallel to the complex situation, it has a unique irreducible quotient which is denoted as $J(\chi_{1},\chi_{2})$;

\item when $\vphi$ is irreducible with parameter $(l,t)\in \BN\times \BC$, it is associated to the discrete series representation
$$
D_{l}\otimes |\det(\cdot)|^{t}
$$
of $\GL_{2}(F)$ \cite[3.4]{knapp55local}.
\end{itemize}

It follows that we can draw the following conclusion about the functorial image of $D_{l}\otimes |\det(\cdot)|^{t}$ along $\rho = \Sym^{n}$. For notational convenience, we write the discrete series representation of $\GL_{2}(F)$ with parameter $(l,t)\in \BN\times \BC$ as $\pi_{(l,t)}$, and use the same notation $(i,t)$ to denote the quasi-character of $\GL_{1}(F)$ with parameter $(i,t)\in \{ \pm 1 \}\times \BC$. We may assume that $t\in i\BR$ where by abuse of notation $i$ is the imaginary number satisfying $i^{2} = -1$ as we only concern about the tempered representations of $\GL_{2}(F)$.

\begin{lem}\label{functorial:real}
The functorial image along $\rho = \Sym^{n}$ of $\pi_{(l,t)}$ is of the following form,
$$
\Ind^{\GL_{n+1}}_{P_{n+1}}
\pi_{M_{n+1}}\otimes 1_{N_{n+1}},
$$
where 
$$
\pi_{M_{n+1}} = 
\bigg\{
\begin{matrix}
\{
\bigotimes_{k=1}^{n/2}
\big((-1)^{l(n/2-k)}, (n-2k)t)\otimes 
\pi_{(2kl, 2kt)}\big)
\}
\bigotimes ((-1)^{nl/2},nt),  & \text{ if $n$ is even},
\\
\bigotimes_{k=1}^{(n+1)/2}
\big((-1)^{l(n-2k+1)/2}, (n-2k+1)t/2)\otimes 
\pi_{((2k-1)l, (2k-1)t)}\big),  & \text{ if $n$ is odd}.\\
\end{matrix}
$$
\end{lem}
\begin{proof}
We compute the representation $\Sym^{n}\circ \vphi$ explicitly with $\vphi = \vphi_{(l,t)}$.

Under the ordered basis $\{ v_{1},v_{2} \}$, we have the following matrix presentations,
$$
\vphi|_{\BC^{\times}}(z) = 
\left(
\begin{matrix}
|z|^{t}(\frac{z}{|z|})^{l} & \\
 & |z|^{t}(\frac{z}{|z|})^{-l}
\end{matrix}
\right),
\quad 
\vphi(j) = 
\left(
\begin{matrix}
	& (-1)^{j} \\
1	&
\end{matrix}
\right).
$$
After composing with $\Sym^{n}$, 
\begin{align*}
&\Sym^{n}\circ \vphi|_{\BC^{\times}}(z) = 
|z|^{nt}\diag ((\frac{z}{|z|})^{nl},(\frac{z}{|z|})^{(n-2)l},...,(\frac{z}{|z|})^{-nl}),
\\
&\Sym^{n}\circ \vphi(j) = 
\text{antidiag}
((-1)^{nl},(-1)^{(n-1)l},...,(-1)^{l},1).
\end{align*}
It follows that the following description for $\Sym^{n}\circ \vphi_{(l,t)}$ holds,
\begin{align*}
\Sym^{n}\circ \vphi_{(l,t)} 
=\bigg\{
\begin{matrix}
\bigoplus_{k=1}^{n/2}
\big(
\vphi_{(2kl,2kt)}\otimes \chi_{(-1)^{l(n/2-k)},(n-2k)t}
\big)
\bigoplus \chi_{(-1)^{nl/2},nt}, & \text{ if $n$ is even},
\\
\bigoplus_{k=1}^{(n+1)/2}
\big(
\vphi_{(2k-1)l,(2k-1)t}
\otimes\chi_{(-1)^{(n+1-2k)l/2}, (n+1-2k)t}
\big)
,
& \text{ if $n$ is odd}.
\end{matrix}
\end{align*}
Here for any $(i,t)\in \{\pm 1 \}\times \BC$, the quasi-character $\chi_{i,t}$ of $W_{F}$ is defined as follows,
$$
\chi_{i,t}(z) = |z|^{t},\quad \chi_{i,t}(j) = i.
$$
From \cite[\S 3]{MR2684298} the associated irreducible admissible representation for $\Sym^{n}\circ \vphi_{(l,t)}$ is fully induced, which is exactly of the form described in the statement.
\end{proof}

\begin{rmk}\label{rmk:parallel}
Comparing Lemma \ref{functorial:real} with Remark \ref{rmk:complex}, we find that the functorial lifting for tempered principal series and discrete series share strong similarities with each other. 

More precisely, as is shown in \cite[2.19]{MR3117742} (which follows from the character table in \cite{MR855239}), for any $l\geq 1$ the discrete series representation $D_{l}$ of $\GL_{2}(F)$ has central character $((-1)^{l},0)$, therefore the central character of $D_{l}\otimes |\det(\cdot)|^{t}$ evaluating at $t\in F^{\times}\hookrightarrow 
\left(
\begin{matrix}
t & \\
& t
\end{matrix}
\right)$ is the same as the character $((-1)^{l},2t)$ evaluating at $t\in F^{\times}$.

When $n$ is even, for any $1\leq k\leq n/2$, the character $(\chi_{1}\chi_{2})^{(n-2k)/2}$ is exactly $(n/2-k)$-th power of the central character of $\Ind_{B_{2}}^{\GL_{2}}(\chi_{1},\chi_{2})$ evaluated at $t\in F^{\times}\hookrightarrow 
\left(
\begin{matrix}
t 	& \\
	&t
\end{matrix}
\right)
$, and $((-1)^{l(n/2-k)}, (n-2k)t)$ is also $(n/2-k)$-th power of the central character of $\pi_{(l,t)}$ evaluated at $t\in F^{\times}\hookrightarrow 
\left(
\begin{matrix}
t 	& \\
	&t
\end{matrix}
\right)
$. Parallel results hold for the case when $n$ is odd.
\end{rmk}

\subsection{$p$-adic case}\label{sub:p-adicfunctorial}
Let $F$ be a $p$-adic field with residue characteristic not equal to $2$. We recall the local Langlands correspondence for $\GL_{2}(F)$ following \cite[\S 8]{gl2llc}, which says that there is a canonical bijection between the set of equivalence classes of irreducible smooth representations of $\GL_{2}(F)$ and the set of equivalence classes of $2$-dimensional representations of $\CW_{F}$ whose restriction to $\SL_{2}(\BC)$ is algebraic and the restriction to $W_{F}$ is continuous with Frobenius semi-simple image (see \cite[\S 2.1]{MR2730575} for the isomorphism between $\CW_{F}$ and the Weil-Deligne group introduced in \cite[\S 7]{gl2llc}).

More precisely, given a $2$-dimensional representation $\vphi$ of $\CW_{F}$ satisfying the above descriptions, the following local Langlands correspondence for $\GL_{2}(F)$ holds.
\begin{itemize}
\item When $\vphi$ is irreducible and $\vphi|_{\SL_{2}(\BC)}$ is trivial, $\vphi$ is associated to a supercuspidal representation of $\GL_{2}(F)$;

\item When $\vphi$ is the tensor product of a quasi-character $\chi$ of $W_{F}$ and the standard representation of $\SL_{2}(\BC)$, $\vphi$ is associated to the twisted Steinberg representation $\chi\otimes \mathrm{St}_{2}$ of $\GL_{2}(F)$, where $\chi$ can be identified with a quasi-character of $F^{\times}$ via $W^{\mathrm{ab}}_{F}\simeq F^{\times}$ \cite{MR0220701}, and $\mathrm{St}_{2}$ is the unique sub-representation appearing in the principal series representation $\Ind_{B_{2}}^{\GL_{2}}(\del^{\frac{1}{2}}_{B_{2}})$;

\item When $\vphi$ is the direct sum of two quasi-characters $\chi_{1},\chi_{2}$ of $W_{F}$ with $\chi_{1}/\chi_{2}\neq |\cdot|^{\pm 1}$, and $\vphi|_{\SL_{2}(\BC)}$ is trivial, $\vphi$ is associated to the irreducible principal series representation $\Ind^{\GL_{2}}_{B_{2}}(\chi_{1},\chi_{2})$ of $\GL_{2}(F)$;

\item When $\vphi$ is the direct sum of two quasi-characters $\chi_{1},\chi_{2}$ of $W_{F}$ with $\chi_{1} = \chi|\cdot|^{1/2}$ and $\chi_{2} = \chi|\cdot|^{-1/2}$, and $\vphi|_{\SL_{2}(\BC)}$ is trivial, $\vphi$ is associated to the $1$-dimensional quasi-character $\chi$ of $\GL_{2}(F)$.
\end{itemize}
We will focus on the tempered representations \cite{MR0499005} of $\GL_{2}(F)$. In particular, the quasi-characters appearing in the above parametrization are always taken to be unitary, and we will not consider the $1$-dimensional representations of $\GL_{2}(F)$.

It turns out that the functorial lifting along $\rho = \Sym^{n}$ of tempered principal series representations of $\GL_{2}(F)$ have the same description as in Lemma \ref{functorial:complex}, and the functorial lifting of the twisted Steinberg representation with unitary character $\chi$, which is denoted as $\chi\otimes \mathrm{St}_{2}$, is given by the twisted Steinberg representation $\chi^{n}\otimes \mathrm{St}_{n+1}$, where $\mathrm{St}_{n+1}$ is the unique sub-representation of $\Ind^{\GL_{n+1}}_{B_{n+1}}(\del^{1/2}_{B_{n+1}})$.

Finally the functorial lifting of the supercuspidal representations of $\GL_{2}(F)$ admits the following description. More precisely, the supercuspidal representations of $\GL_{2}(F)$ are parametrized by the admissible pairs $(E/F,\theta)$ defined in \cite[18.2]{gl2llc}, where $E$ is a quadratic extension of $F$, and $\theta$ is an admissible character of $E^{\times}$. We denote the associated supercuspidal representation by $\pi_{\theta}$. Then the Langlands parameter $\vphi$ for $\pi_{\theta}$ is of the form 
$$
\vphi  = \Ind_{W_{E}}^{W_{F}}\theta,
$$
where $W_{E}$ is the Weil group of $E$. From the isomorphism $W_{E}^{\mathrm{ab}}\simeq E^{\times}$ \cite{MR0220701} the quasi-character $\theta$ of $E^{\times}$ can be identified with a quasi-character of $W_{E}$. Let the associated quotient map be $a_{E}:W_{E}\to W^{\ab}_{E}\simeq E^{\times}$.

For notational convenience, when $\theta$ is not an admissible character of $E^{\times}$, we will still denote by $\pi_{\theta}$ the irreducible admissible representation of $\GL_{2}(F)$ associated to the Langlands parameter $\Ind_{W_{E}}^{W_{F}}\theta$.

Let $W_{F}/W_{E}\simeq \Gal(E/F) = \{ 1,\lam\}$, where $\lam$ is the nontrivial involution defined on $E$ fixing $F$. Let $\RN_{E/F}$ be the norm map.

Then the following description for the functorial image of $\pi_{\theta}$ along $\Sym^{n}$ holds.

\begin{lem}\label{functorial:padic}
The functorial image along $\rho =\Sym^{n}$ of $\pi_{\theta}$ is of the following form,
$$
\Ind^{\GL_{n+1}}_{P_{n+1}}
\pi_{M_{n+1}}\otimes 1_{N_{n+1}},
$$
where 
$$
\pi_{M_{n+1}} = \bigg\{
\begin{matrix}
\bigotimes_{k=1}^{n/2}
(
\theta^{(n-2k)/2}\otimes \pi_{\theta^{2k}}
)
\bigotimes (\theta)^{n/2} &\text{ if $n$ is even},\\
\bigotimes_{k=1}^{(n+1)/2}
(\theta^{(n-2k+1)/2}\otimes \pi_{\theta^{2k-1}})
, & \text{ if $n$ is odd}.
\end{matrix}
$$
Here when taking tensor product we abuse the notation viewing $\theta$ as a quasi-character of $F^{\times}$ via restriction.
\end{lem}
\begin{proof}
The Langlands parameter $\vphi = \Ind^{W_{F}}_{W_{E}}\theta$ admits the following description. We fix an element $\wt{\lam}\in W_{F}\bs W_{E}$ such that $\wt{\lam}^{2} = \lam_{E}\in W_{E}$. Then $\vphi$ has an ordered basis $\{v_{1},v_{2} \}$, where $v_{1}$ is the function on $W_{F}$ defined by $v_{1}(1) = 1$ and $v_{1}(\wt{\lam}) = 0$, and $v_{2}$ is the function on $W_{F}$ defined by $v_{2}(1) = 0$ and $v_{2}(\wt{\lam}) = 1$. Under the ordered basis, $\vphi$ has the following matrix presentation,
$$
\vphi|_{W_{E}} = 
\left(
\begin{matrix}
\theta\circ a_{E} & \\
 & \theta^{\lam}\circ a_{E}
\end{matrix}
\right),\quad 
\vphi(\wt{\lam}) = 
\left(
\begin{matrix}
	&\theta\circ a_{E}(\lam_{E})\\
1	&
\end{matrix}
\right)
$$
where $\theta^{\lam} = \theta\circ \lam$. For any $w_{e}\in W_{E}$, we have $\theta^{\lam}\circ a_{E}(w_{e})  =\theta\circ a_{E}(\wt{\lam}w_{e}\wt{\lam}^{-1})$.

Therefore after composing with $\Sym^{n}$, 
\begin{align*}
&\Sym^{n}\circ \vphi|_{W_{E}} = 
\diag(\theta^{n}\circ a_{E}, \theta^{n-1}\theta^{\lam}\circ a_{E},...,(\theta^{\lam})^{n}\circ a_{E}),
\\
&
\Sym^{n}
\circ \vphi(\wt{\lam})
=\mathrm{antidiag}
(\theta^{n}\circ a_{E}(\lam_E), \theta^{n-1}\circ a_{E}(\lam_{E}),...,1).
\end{align*}
It follows that the following description for the representation $\Sym^{n}\circ \vphi$ holds.

\begin{itemize}
\item When $n$ is odd, for any integer $0\leq k\leq \frac{n-1}{2}$, we consider the following representation $\vphi_{k}$ of $W_{F}$,
$$
\vphi_{k}|_{W_{E}}
=
\left(
\begin{matrix}
\theta^{n-k}(\theta^{\lam})^{k}\circ a_{E} & \\
& \theta^{k}(\theta^{\lam})^{n-k}\circ a_{E}
\end{matrix}
\right),
\quad 
\vphi_{k}(\wt{\lam}) = 
\left(
\begin{matrix}
	&\theta^{n-k}\circ a_{E}(\lam_{E})\\
\theta^{k}\circ a_{E}(\lam_{E})&
\end{matrix}
\right).
$$
Then $\Sym^{n}\circ \vphi \simeq \bigoplus_{k=0}^{\frac{n-1}{2}}\vphi_{k}$. 

Here we notice that under the ordered basis $\{v_{1}^{k},v_{2}^{k}\}$ given by
$$
v_{1}^{k}(1) = 1, v_{1}^{k}(\wt{\lam}) =0, \quad v_{2}^{k}(1)=0, v^{k}_{2}(\wt{\lam}) = \theta^{-k}\circ a_{E}(\lam_{E}), 
$$
the representation 
$\vphi_{k}$ is isomorphic to $\Ind^{W_F}_{W_{E}}\theta^{n-k}(\theta^{\lam})^{k}$. In particular
$$
\Sym^{n}\circ \vphi \simeq  
\bigoplus_{k=0}^{\frac{n-1}{2}}
\Ind^{W_F}_{W_{E}}\theta^{n-k}(\theta^{\lam})^{k}.
$$

\item When $n$ is even, we consider the following quasi-character of $W_{F}$,
$$
\chi|_{W_{E}} = 
\theta^{n/2}(\theta^{\lam})^{n/2}\circ a_{E}, 
\quad 
\chi(\wt{\lam}) = 
\theta^{n/2}\circ a_{E}(\lam_{E}).
$$
Then $\Sym^{n}\circ \vphi= \bigoplus_{k=0}^{n/2}\vphi_{k}
\bigoplus \chi$.

We refine the description for $\chi$ as follows.

Recall the following commutative diagram from local class field theory 
\cite{MR0220701}
$$
\xymatrix{
W_{E}\ar[d] \ar[r]^{a_{E}}& E^{\times}\ar[d]^{\RN_{E/F}}	\\	
W_{F} \ar[r]^{a_{F}} & F^{\times}
}
$$
It turns out that as a quasi-character of $F^{\times}$, $\chi$ is actually equal to the restriction of $(\theta)^{n/2}$ to $F^{\times}$. 

More precisely, after composing with $a_{F}$ the quasi-character $(\theta)^{n/2}$ of $F^{\times}$ can be viewed as a quasi-character of $W_{F}$. From the commutativity of the diagram, for any $w_{e}\in W_{E}$,
\begin{align*}
&(\theta)^{n/2}\circ a_{F}(w_{e}) = 
(\theta)^{n/2}(\RN_{E/F}\circ(a_{E}(w_{e})))
\\
=
&(\theta)^{n/2}\circ a_{E}(w_{e})
(\theta^{\lam})^{n/2}\circ (a_{E}(w_{e}))
=\chi(w_{e}),
\end{align*}
and 
$$
\chi(\wt{\lam}) = 
(\theta)^{n/2}\circ a_{E}(\lam_{E}) =
(\theta^{n/2})(\RN_{E/F}\circ a_{E}(\lam_{E}))
=
(\theta)^{n/2}
\circ a_{F}(\wt{\lam}).
$$
Therefore $\chi = (\theta)^{n/2}$ as a quasi-character of $F^{\times}$.
\end{itemize}

Now we make the following observation. For any pair of non-negative integers $(a,b)$ with $a+b = n$ and $a>b\geq 0$,
$$
\Ind^{W_{F}}_{W_{E}}
\theta^{a}(\theta^{\lam})^{b}
=\Ind^{W_{F}}_{W_{E}}(\theta\theta^{\lam})^{b}\theta^{a-b}.
$$
The quasi-character $\theta\theta^{\lam}$ is invariant under the action of $\Gal(E/F)$, and it can be written as $\theta\circ \RN_{E/F}$ where we abuse the notation viewing $\theta$ as a quasi-character of $F^{\times}$ via restriction. Therefore from \cite[34.4]{gl2llc} the following isomorphism holds
$$
\Ind^{W_{F}}_{W_{E}}(\theta\circ \RN_{E/F})^{b}\theta^{a-b}
\simeq (\theta)^{b}\otimes \Ind^{W_{F}}_{W_{E}}\theta^{a-b}.
$$
Here again when taking tensor product we view $\theta$ as a quasi-character of $F^{\times}$ via restriction.

Therefore the following description holds
$$
\Sym^{n}\circ \vphi = 
\bigg\{
\begin{matrix}
\bigoplus_{k=1}^{n/2}
(
\theta^{(n-2k)/2}\otimes \Ind^{W_{F}}_{W_{E}}\theta^{2k} 
)
\bigoplus (\theta)^{n/2} &\text{ if $n$ is even},\\
\bigoplus_{k=1}^{(n+1)/2}
\theta^{(n-2k+1)/2}\otimes \Ind^{W_{F}}_{W_{E}}
\theta^{2k-1}
, & \text{ if $n$ is odd}.
\end{matrix}
$$
From \cite[\S 3]{MR0499005} the associated irreducible admissible representation for $\Sym^{n}\circ \vphi$ is fully induced, which is exactly the one appearing in the statement.
\end{proof}

\begin{rmk}\label{rmk:similarpadic}
Comparing with the decomposition of the discrete series representations in Lemma \ref{functorial:real}, we can observe that the description of functorial images in the $p$-adic case has strong similarities with the Archimedean case, where we use the fact that for an admissible pair $(E/F,\theta)$, the supercuspidal representation $\pi_{\theta}$ has central character $\theta$ \cite[19.1.1]{gl2llc}. Unfortunately we are not able to unify the twisted Steinberg representations.
\end{rmk}

\begin{rmk}\label{rmk:non-admissible}
When $\theta = \theta^{\lam}$, i.e. $\theta =\xi\circ \RN_{E/F}$ for some quasi-character $\xi$ of $F^{\times}$, following identification holds
$$
\Ind_{W_{F}}^{W_{E}}\theta = \Ind_{W_{E}}^{W_{F}}\xi\circ \RN_{E/F} = \xi\otimes\Ind^{W_{F}}_{W_{E}}1,
$$
where $1$ is denoted as the trivial representation of $W_{E}$. In particular, the representation
$
\Ind^{W_{F}}_{W_{E}}1
$
can be written as the direct sum of two characters of $W_{F}$,
$$
1\oplus \sgn_{E/F},
$$
where $\sgn_{E/F}$ is the quadratic character of $F^{\times}$ associated to $E$.

Therefore the irreducible admissible representation of $\GL_{2}(F)$ associated to $\Ind^{W_{F}}_{W_{E}}\theta$ is the irreducible principal series representation
$$
\xi\otimes \Ind^{\GL_{2}}_{B_{2}}
(1\otimes \sgn_{E/F}).
$$
\end{rmk}

\section{Stable transfer}\label{sec:stabletransfer}

In this section, we are going to establish Theorem \ref{thm:characterrelation}, which is the main result of the paper. From Theorem \ref{thm:characterrelation} the construction of the invariant distribution $\Theta^{\rho}$ claimed in Corollary \ref{cor:distributiongroup} and Theorem \ref{thm:distributioncomplex} are not hard to derive following the discussion in the introduction. We divide the proof into three parts, depending on whether the tempered representations are principal series, discrete series in the real case, or supercuspidal representations in the $p$-adic case.

\subsection{Principal series}\label{subsec:ps}
In this section we treat the tempered principal series representations.

\subsubsection{Explicit formula}
Let $F$ be any local field of characteristic zero. For any two unitary characters $\chi_{1},\chi_{2}$ of $F^{\times}$, let $\pi = \Ind^{\GL_{2}}_{B_{2}}(\chi_{1},\chi_{2})$ be the tempered principal series representation of $\GL_{2}(F)$. From Lemma \ref{functorial:complex}, the lifting along $\rho = \Sym^{n}$ of $\pi$ from $\GL_{2}(F)$ to $\GL_{n+1}(F)$, which is denoted as $\Pi$, is isomorphic to the induced representation $\Ind_{B_{n+1}}^{\GL_{n+1}}(\chi^{n}_{1},\chi^{n-1}_{1}\chi_{2},...,\chi^{n}_{2})$. From the character formula for fully induced representations \cite[Theorem~3]{Gv72} \cite[(10.27)]{MR855239}, both $\tr_{\pi}$ and $\tr_{\Pi}$ are determined by their restrictions on $T_{2}(F)^{\rss}$ and $T_{n+1}(F)^{\rss}$, and the explicit formulas are given as follows,
\begin{align*}
D^{1/2}_{\GL_{2}}(t)\tr_{\pi}(t) = &\sum_{s\in S_{2}}
\chi_{1}(t_{s(1)})\chi_{2}(t_{s(2)}), \quad t= \diag(t_{1},t_{2})\in T_{2}^{\rss},
\\
D^{1/2}_{\GL_{n+1}}(T)\tr_{\Pi}(T) = &\sum_{s\in S_{n+1}}\chi^{n}_{1}(T_{s(1)})\chi^{n-1}_{1}\chi_{2}(T_{s(2)})...\chi^{n}_{2}(T_{s(2)}), 
\\
&T=  \diag(T_{1},...,T_{n+1})\in T^{\rss}_{n+1}.
\end{align*}
Here $S_{n}$ is the permutation group of the set $\{1,2,...,n \}$. On the other hand, the following identity holds,
$$
\sum_{s\in S_{n+1}}\chi^{n}_{1}(T_{s(1)})\chi^{n-1}_{1}\chi_{2}(T_{s(2)})...\chi^{n}_{2}(T_{s(2)})
=
\sum_{s\in S_{n+1}}
\chi_{1}(\prod_{k=1}^{n+1}T^{n+1-k}_{s(k)})\chi_{2}(\prod_{k=1}^{n+1}T^{k-1}_{s(k)}).
$$
It follows from straightforward computation that the distribution $\Theta^{\rho}$ defined in Theorem \ref{thm:distributioncomplex} yields the distributional identity for any $\pi$ a tempered principal series of $\GL_{2}(F)$.
$$
D^{1/2}_{\GL_{n+1}}(g)
\tr_{\Pi}(g) = \int_{\GL_{2}(F)}
\Theta^{\rho}(g,\gam)D^{1/2}_{\GL_{2}}
(\gam)\tr_{\pi}(\gam), \quad (g,\gam)\in \GL_{n+1}^{\rss}(F)\times \GL_{2}(F)^{\rss}.
$$
To make a brief summary, we recall the statement of Theorem \ref{thm:distributioncomplex} in the following.

\begin{thm}\label{thm:distributioncomplex:content}
Define the distribution $\Theta^{\rho}$ on $\GL_{n+1}(F)^{\rss}\times \GL_{2}(F)^{\rss}$ as follows. It is invariant under the conjugation action of $\GL_{n+1}(F)\times \GL_{2}(F)$, and its formula on the diagonal maximal split torus is given by
\begin{equation}\label{eq:distributiontorus}
\Theta^{\rho}(T,t) = 
\frac{1}{2}
\sum_{w\in S_{n+1}}\del_{\prod_{k=1}^{n+1}T^{n+1-k}_{w(k)}}(t_{1})\otimes \del_{\prod_{k=1}^{n+1}T^{k-1}_{w(k)}}(t_{2}),
\end{equation}
where $T = (T_{1},...,T_{n+1})$ (resp. $t = (t_{1},t_{2})$) is a diagonal regular semi-simple element in $\GL_{n+1}(F)$ (resp. $\GL_{2}(F)$), $\del_{*}(*)$ is the delta distribution, and $S_{n+1}$ is the permutation group of the set $\{1,...,n+1 \}$.

Then the following distributional identity holds
$$
D^{1/2}_{\GL_{n+1}}(g) 
\tr_{\Ind_{B_{n+1}}^{\GL_{n+1}}(\chi^{n}_{1},\chi^{n-1}_{1}\chi_{2},...,\chi^{n}_{2})}
(g) 
=\int_{\GL_{2}(F)}
\Theta^{\rho}(g,\gam)
D^{1/2}_{\GL_{2}}(\gam)\tr_{\Ind^{\GL_{2}}_{B_{2}}(\chi_{1},\chi_{2})}(\gam)
$$
for any $(g,\gam)\in \GL_{n+1}(F)^{\rss}\times \GL_{2}(F)^{\rss}$.
\end{thm}
We note that $\Theta^{\rho}$ can also be viewed as the parabolic induction of the distribution (\ref{eq:distributiontorus}) from $T_{n+1}(F)\times T_{2}(F)$ to $\GL_{n+1}(F)\times \GL_{2}(F)$ following the notation in \cite[(3.4.2)]{beuzart2015local}.

\subsubsection{Reduction formula}

On the other hand, as we have already seen in the introduction, with the character formula for fully induced representations \cite[Theorem~3]{Gv72} \cite[(10.27)]{MR855239}, for any tempered representation $\pi$ of $\GL_{2}(F)$, if we denote the lifting of $\pi$ along $\rho = \Sym^{n}$ by $\Pi$, except the twisted Steinberg representations when $F$ is a $p$-adic field, the character of $\Pi$ is always supported on the $\GL_{n+1}(F)$-conjugacy classes of regular semi-simple elements in $M_{n+1}(F)$, which in turn is determined by the character formula of $\pi_{M_{n+1}}$.
Therefore we may first determine the stable transfer factor from any tempered representation $\pi$ of $\GL_{2}(F)$ to $\pi_{n+1}$ of $M_{n+1}(F)$.

From Remark \ref{rmk:complex} and \ref{rmk:principalseries}, we have
$$
\pi_{M_{n+1}} =
\bigg\{
\begin{matrix}
\big\{
\bigotimes_{k=1}^{n/2}
\big((\chi_{1}\chi_{2})^{(n-2k)/2}
\otimes \Ind_{B_{2}}^{\GL_{2}}(\chi^{2k}_{1},\chi^{2k}_{2})
\big) \big\}\bigotimes (\chi_{1}\chi_{2})^{n/2}, & \text{ if $n$ is even},\\
\bigotimes_{k=1}^{(n+1)/2}\big(
(\chi_{1}\chi_{2})^{(n-2k+1)/2}
\otimes \Ind_{B_{2}}^{\GL_{2}}(\chi^{2k-1}_{1},\chi^{2k-1}_{2})
\big) , & \text{ if $n$ is odd}.
\end{matrix}.
$$

On the other hand, the character of $\pi$ is supported on the $\GL_{2}(F)$-conjugacy classes of the maximal split torus $T_{2}(F)$ inside $\GL_{2}(F)$ given by
$$
D^{1/2}_{\GL_{2}}(t)\tr_{\pi}(t) = \sum_{s\in W_{2}}\chi_{1}(t_{s(1)})\chi_{2}(t_{s(2)}), \quad t= \diag(t_{1},t_{2})\in T_{2}(F)^{\rss}.
$$
It follows that for any positive integer $k\in \BN$, 
$$
D_{\GL_{2}}^{1/2}(t)\tr_{\Ind^{\GL_{2}}_{B_{2}}(\chi_{1}^{k},\chi_{2}^{k})}(t)
=D_{\GL_{2}}^{1/2}(t^{k})\tr_{\Ind^{\GL_{2}}_{B_{2}}
(\chi_{1},\chi_{2})}(t^{k}), \quad t= \diag(t_{1},t_{2})\in T_{2}(F)^{\rss}.
$$
Since both sides of the equality are invariant under the conjugation action of $\GL_{2}(F)$, the desired distributional equality claimed in Theorem \ref{thm:characterrelation} holds,
$$
D^{1/2}_{\GL_{2}}(\gam^{k})
\tr_{\Ind^{\GL_{2}}_{B_{2}}(\chi_{1},\chi_{2})}
(\gam^{k}) = D^{1/2}_{\GL_{2}}(\gam)
\tr_{\Ind_{B_{2}}^{\GL_{2}}(\chi^{k}_{1},\chi^{k}_{2})}(\gam), \quad \gam\in \GL_{2}(F)^{\rss}.
$$
\begin{rmk}\label{rmk:zeroomit}
It can happen that for some element $g\in \GL_{2}(F)^{\rss}$ which does not lie in the $\GL_{2}(F)$-conjugacy class of $T_{2}(F)$, $g^{k}$ lies in $T_{2}(F)^{\rss}$ for some $k$. A typical example is when $k=2$ and $F=\BR$. Let $g=  
\left(
\begin{matrix}
\frac{1}{\sqrt{2}} & \frac{1}{\sqrt{2}} \\
-\frac{1}{\sqrt{2}} & \frac{1}{\sqrt{2}}
\end{matrix}
\right).$ Then 
$g^{2} = 
\left(
\begin{matrix}
0	& 1\\
1   & 0
\end{matrix}
\right)
$
is diagonalizable over $\BR$ with eigenvalues $1$ and $-1$. But it won't affect the distribution that we are considering. This is because these elements have measure zero in $\GL_{2}(F)^{\rss}$ since up to conjugation $g^{k}$ lies in the intersection of the $\GL_{2}(F)$-conjugacy classes of split torus $T_{2}(F)$ and the elliptic toruses, whose $\GL_{2}(F)$-conjugacy classes have measure zero through comparing their eigenvalues. The same discussion applies to other cases below.
\end{rmk}

It follows from straightforward computation that the distribution $\Theta^{M_{n+1}}$ constructed in Corollary \ref{cor:reducetoDel} based on Assumption \ref{assum} yields the lifting of the distribution character of $\pi$ to that of $\pi_{M_{n+1}}$ for any $\pi$ a tempered principal series.

\begin{thm}\label{lift:principal}
When $n$ is odd, define the distribution $\Theta^{M_{n+1}}$ as follows. For any element $(g_{1},...,g_{(n+1)/2})\in \GL_{2}(F)\times...\GL_{2}(F)^{\rss}$ and $h\in \GL_{2}(F)^{\rss}$, let
\begin{align*}
&\Theta^{M_{n+1}}
((g_{1},...,g_{(n+1)/2}),h) = 
\\
&\Theta^{\Del_{(n+1)/2}}
(\det(g_{1})^{(n-1)/2}\cdot g_{1},\det(g_{2})^{(n-3)/2}\cdot g_{2},..., \det(g_{(n+1)/2})\cdot g_{(n-1)/2}, g_{(n+1)/2},h).
\end{align*}
Then the following distributional identity holds
\begin{align*}
&D^{1/2}_{M_{n+1}}
(g_{1},...,g_{(n+1)/2})
\tr_{\pi_{M_{n+1}}}(g_{1},...,g_{(n+1)/2})
\\
&=\int_{\GL_{2}(F)^{\rss}}\Theta^{M_{n+1}}
(g_{1},...,g_{(n+1)/2},h)D^{1/2}_{\GL_{2}}(h)\tr_{\pi}(h)
\end{align*}
for any tempered principal series $\pi$ of $\GL_{2}(F)$.
\end{thm}

\begin{rmk}\label{rmk:evenprincipal}
There is a similar formula when $n$ is even that is also straightfoward to see. More precisely, for any element $(g_{1},...,g_{n/2},a)\in \GL_{2}(F)^{\rss}\times...\times \GL_{2}(F)^{\rss}\times \GL_{1}(F)$ and $h\in \GL_{2}(F)$, let
\begin{align*}
&\Theta^{M_{n+1}}((g_{1},...,g_{n/2},a),h)
\\
=
&
\Theta^{\Del_{n/2}}(a\cdot \det(g_{1})^{(n-2)/2}\cdot g_{2}^{2},a\cdot \det(g_{1})^{(n-4)/2}\cdot g_{2}^{4},...,a\cdot g_{2}^{2n},h).
\end{align*}
Here parallel to Assumption \ref{assum}, $\Theta^{\Del_{n/2}}$ is assumed to be the distribution giving rise to the identity (\ref{eq:langlandsconjecture}) for the $L$-morphism given by the diagonal embedding 
$\GL_{2}(\BC)\to (\GL_{2}(\BC))^{n/2}$.
Then the following distributional identity holds
\begin{align*}
&D^{1/2}_{M_{n+1}}
((g_{1},...,g_{n/2},a))
\tr_{\pi_{M_{n+1}}}((g_{1},...,g_{n/2},a))
\\
&=\int_{\GL_{2}(F)^{\rss}}\Theta^{M_{n+1}}
((g_{1},...,g_{n/2},a),h)D^{1/2}_{\GL_{2}}(h)\tr_{\pi}(h).
\end{align*}

Unfortunately, the formula for $\Theta^{M_{n+1}}$ only works for tempered principal series. More precisely, as we will see in later sections, for discrete series representations when $F=\BR$ or supercuspidal representations when $F$ is $p$-adic, one needs slight modifications for $\Theta^{M_{n+1}}$, which in particular does not give us a uniform formula for both tempered principal series representations and discrete series representations.
\end{rmk}

\subsection{Discrete series representations: Real case}\label{sub:discreteseries}\label{subsec:discreteseries}

Let $F=\BR$. For any $(l,t)\in \BN\times \BC$, let $\pi = \pi_{(l,t)}$ be the discrete series representation $D_{l}\otimes |\det(\cdot)|^{t}_{\BR}$ of $\GL_{2}(F)$. From Lemma \ref{functorial:real},
$$
\pi_{M_{n+1}} = 
\bigg\{
\begin{matrix}
\{
\bigotimes_{k=1}^{n/2}
\big((-1)^{l(n/2-k)}, (n-2k)t
)\otimes 
\pi_{(2kl, 2kt)}
\big)
\}
\bigotimes ((-1)^{nl/2},nt),  & \text{ if $n$ is even},
\\
\bigotimes_{k=1}^{(n+1)/2}
((-1)^{l(n-2k+1)/2}, (n-2k+1)t/2)\otimes 
\pi_{((2k-1)l, (2k-1)t)},  & \text{ if $n$ is odd}.\\
\end{matrix}
$$
Similar to the situation of tempered principal series, for any $k\in \BN$, we are going to find the relation between the distribution characters of $\pi_{(l,t)}$ and $\pi_{(kl,kt)}$.

We first recall the character formula for discrete series representation $D_{l}$ of $\GL_{2}$. Following \cite[\S 2]{knapp55local}, let $\SL^{\pm}_{2}(F)$ be the subgroup of $\GL_{2}(F)$ consisting of elements with determinant $\{\pm 1 \}$. Let $D_{l}^{+}$ be the holomorphic discrete series representation of $\SL_{2}(F)$. As a representation of $\SL^{\pm}_{2}(F)$,
$$
D_{l} = \Ind^{\SL_{2}(F)^{\pm}}_{\SL_{2}(F)}D^{+}_{l}.
$$
The restriction of $D_{l}$ to $\SL_{2}(F)$ decomposes as the direct sum of the holomorphic discrete series $D^{+}_{l}$ and the anti-holomorphic discrete series $D^{-}_{l}$. From \cite[2.19]{MR3117742} (or \cite[Corollary~10.13]{MR855239}) the following character formula holds for $D_{l}$ when restricted to $\SL_{2}(F)$,
\begin{align*}
&D^{\frac{1}{2}}_{\GL_{2}}(a(\alp))\tr_{D_{l}}(\pm a(\alp)) = (\pm 1)^{l}
(-2e^{-l\alp}), \quad \pm a(\alp) = \pm 
\left(
\begin{matrix}
e^{\alp}	&	\\
		&e^{-\alp}
\end{matrix}
\right), \alp>0,
\\
&D^{\frac{1}{2}}_{\GL_{2}}(s(\theta))\tr_{D_{l}}(s(\theta)) =
-(e^{il\theta} - e^{-il\theta}), \quad
s(\theta) = 
\left(
\begin{matrix}
\cos \theta & \sin \theta\\
-\sin \theta & \cos \theta
\end{matrix}
\right).
\end{align*}
On the other hand, any elements in $\SL^{\pm}_{2}(F)$ with negative determinant interchange $D_{l}^{+}$ and $D^{-}_{l}$. Therefore elements with negative determinant have trace zero. It follows that when viewing $D_{l}$ as a representation of $\GL_{2}(F)$, the distribution character of $D_{l}$ is supported only on elements with positive determinant, which also has trivial central character, and its value on $\SL_{2}(F)$ is given as above.

Therefore for any $k\in \BN$,
\begin{align*}
&D^{\frac{1}{2}}_{\GL_{2}}(a(\alp))\tr_{D_{kl}}(\pm a(\alp)) = (\pm 1)^{kl}(-2e^{-kl\alp})
=D^{\frac{1}{2}}_{\GL_{2}}((a(\alp))^{k})
\tr_{D_{l}}((\pm a(\alp))^{k})
,
\\
&D^{\frac{1}{2}}_{\GL_{2}}(s(\theta))
\tr_{D_{kl}}(s(\theta)) = 
-(e^{ikl\theta} - e^{-ikl\theta})
=D^{\frac{1}{2}}_{\GL_{2}}((s(\theta))^{k})
\tr_{D_{l}}((s(\theta))^{k}).
\end{align*}
Again both sides are invariant under the conjugation action of $\GL_{2}(F)$, hence the following distributional identity holds for any $k\in \BN$ and $g\in \GL_{2}(F)^{\rss}$,
\begin{align*}
\bigg\{
\begin{matrix}
D^{1/2}_{\GL_{2}}(g)\tr_{D_{kl}}(g) = D^{1/2}_{\GL_{2}}(g^{k})\tr_{D_{l}}(g^{k}), & \text{ when $k$ is odd,}
\\
D^{1/2}_{\GL_{2}}(g)\tr_{D_{kl}}(g) = D^{1/2}_{\GL_{2}}(g^{k})\tr_{D_{l}}(g^{k})\frac{\sgn(\det g)+1}{2}, & \text{ when $k$ is even}.
\end{matrix}
\end{align*}
Here $\sgn$ is the sign character of $F^{\times}$.
\begin{rmk}\label{rmk:discreteseriescoincidence}
The reason for the additional factor $\frac{\sgn(\det g)+1}{2}$ when $k$ is even is that, when $\det g$ is negative, $\tr_{D_{kl}}(g)$ is equal to $0$, but $g^{k}$ has positive determinant, which in particular implies that $\tr_{D_{l}}(g^{k})$ is nonzero.
\end{rmk}

Combing with Remark \ref{rmk:parallel} concerning the central characters appearing in $\pi_{M_{n+1}}$, for any $\pi$ a discrete series representation of $\GL_{2}(F)$, the parallel formula appearing in Theorem \ref{lift:principal} yields the lifting of the distribution character from $\pi$ to $\pi_{M_{n+1}}$.

\begin{thm}\label{lift:discrete}
When $n$ is odd, for any element $(g_{1},...,g_{(n+1)/2})\in \GL_{2}(F)^{\rss}\times...\times \GL_{2}(F)^{\rss}$ and $h\in \GL_{2}(F)^{\rss}$, recall the distribution $\Theta^{M_{n+1}}$ defined in Theorem \ref{thm:distributioncomplex:content}. Then the following distributional identity holds
\begin{align*}
&D^{1/2}_{M_{n+1}}
(g_{1},...,g_{(n+1)/2})
\tr_{\pi_{M_{n+1}}}(g_{1},...,g_{(n+1)/2})
\\
=
&\int_{\GL_{2}(F)^{\rss}}\Theta^{M_{n+1}}
(g_{1},...,g_{(n+1)/2},h)D^{1/2}_{\GL_{2}}(h)\tr_{\pi}(h)
\end{align*}
for $\pi$ a tempered discrete series representation of $\GL_{2}(F)$.
In particular, when $n$ is odd, we obtain a uniform formula giving rise to the lifting of distribution characters of tempered representations of $\GL_{2}(F)$ to $M_{n+1}(F)$.
\end{thm}

\begin{rmk}\label{rmk:forevenrealcasediscreteseries}
When $n$ is even, by straightforward computation, for any element $(g_{1},...,g_{n/2},a)\in \GL_{2}(F)^{\rss}\times...\times \GL_{2}(F)^{\rss}\times \GL_{1}(F)$ and $h\in \GL_{2}(F)^{\rss}$, let
\begin{align*}
\Theta^{M_{n+1}}_{D}((g_{1},...,g_{n/2},a),h)
=
\Theta^{M_{n+1}}(g_{1},...,g_{n/2},h)\prod_{k=1}\frac{\sgn(g_{k})+1}{2},
\end{align*}
where the distribution $\Theta^{M_{n+1}}$ is defined in Remark \ref{rmk:evenprincipal}, then the following distributional identity holds,
\begin{align*}
&D^{1/2}_{M_{n+1}}
((g_{1},...,g_{n/2},a))
\tr_{\pi_{M_{n+1}}}((g_{1},...,g_{n/2},a))
\\
&=\int_{\GL_{2}(F)^{\rss}}\Theta^{M_{n+1}}_{D}
((g_{1},...,g_{n/2},a),h)D^{1/2}_{\GL_{2}}(h)\tr_{\pi}(h)
\end{align*}
for $\pi$ a tempered discrete series representation of $\GL_{2}(F)$.
Unfortunately $\Theta^{M_{n+1}}_{D}\neq \Theta^{M_{n+1}}$ unless we restrict $g_{i}$ to $Z_{\GL_{2}}(F)\SL_{2}(F)$ for any $1\leq i\leq n/2$.
\end{rmk}

\subsection{Supercuspidal representations}\label{subsec:supercuspidal}

Let $F$ be a $p$-adic field with residue characteristic not equal to $2$. We are going to explore the parallel phenomenon for supercuspidal representations of $\GL_{2}(F)$.

As we have seen from the discussion of Section \ref{subsec:ps} and Section \ref{subsec:discreteseries}, for a supercuspidal representation $\pi_{\theta}$ of $\GL_{2}(F)$ associated to an admissible pair $(E/F,\theta)$, the key point is to explore the relation between the distribution characters of $\pi_{\theta}$ and $\pi_{\theta^{k}}$ for any $k\in \BN$. The main goal of the present section is to establish Theorem \ref{thm:characterrelation} for the case when $\pi$ is a supercuspidal representation of $\GL_{2}(F)$.

We first review some basics about the group $\GL_{2}(F)$.

Since we assume that $p\neq 2$, every maximal elliptic torus in $\GL_{2}(F)$ is always tamely ramified. Let $\mu_{F}\subset \CO^{\times}_{F}$ be the set of roots of unity lying within $F$ with generator $\eps = \eps_{F}$. Then the set of $\GL_{2}(F)$-conjugacy classes of maximal toruses in $\GL_{2}(F)$ is parametrized by the set $Q_{F} = \{1,\eps,\vpi,\eps\vpi \}$. Any quadratic \'etale algebra $E$ of $F$ is isomorphic to $E_{\alp}$ for some $\alp\in Q_{F}$ such that $E_{\alp}\simeq F(\sqrt{\alp})=F[x]/(x^{2}-\alp)$ for some $\alp\in Q_{F}$ with $\alp\neq 1$, or $E_{1} = F\times F$. For any $\alp\in Q_{F}$, there exists a maximal torus $T_{E_{\alp}}\subset \GL_{2}(F)$ such that $T_{E_{\alp}}\simeq \Res_{E_{\alp}/F}\GL_{1}$. We may pick up the following representative for $T_{E_{\alp}}$ with explicit matrix presentation described as follows. Set 
$$
T_{E_{\alp}}(F) = 
\{
\left(
\begin{matrix}
a & b\alp\\
b & \alp
\end{matrix}
\right)
|\quad a,b\in F, a\neq 0 \text{ or } b\neq 0
\}
$$
for $\alp\neq 1$, and 
$$
T_{E_{1}}(F) = 
\{
\left(
\begin{matrix}
a & 0\\
0 & b
\end{matrix}
\right)	
|\quad
a,b\in F^{\times}
\}.
$$
The matrix presentation ensures that the maximal compact subgroup of $T_{E_{\alp}}(F)$ is contained in $\GL_{2}(\CO_{F})$ for any $\alp\in Q_{F}$. We can verify by definition that the trace map $\mathrm{Tr}_{E_{\alp}/F}$ and norm map $\RN_{E_{\alp}/F}$ on $T_{E_{\alp}}\simeq E^{\times}_{\alp}$ can be identified with the trace and determinant map on $\GL_{2}(F)$ restricted to the maximal torus $T_{E_{\alp}}$.

Then we are going to recall the notion of depth defined for elements and representations of $\GL_{2}(F)$ (which are actually defined for more general reductive $p$-adic groups). 
Let $\CB(\GL_{2},F)$ be the Bruhat-Tits building of $\GL_{2}(F)$ defined in \cite[7.4]{MR0327923}. By \cite{MR1253198} \cite{MR1371680} there exists for each $x\in \CB(\GL_{2},F)$ a decreasing filtration $\{(\GL_{2})_{x,r} \}_{r\in \BR_{\geq 0}}$ of compact open subgroups of $\GL_{2}(F)$. The following definition is standard.

\begin{defin}\label{defin:depth}
For any $\alp\in Q_{F}$ and $\gam\in T_{E_{\alp}}(F)^{\rss}$ with compact image in $T_{E_{\alp}}(F)/Z_{\GL_{2}}(F)$, the depth $d(\gam)\in \BR_{\geq 0}\cup \{ \infty\}$ of $\gam$ is defined to be 
$$
d(\gam) = 
\max
\{
\{ 0\}
\cup 
\{
r\geq 0|\quad
\gam\in (\GL_{2})_{x,r} \text{ for some $x\in \CB(\GL_{2},F)$}
\}
\}
.
$$
\end{defin}

It turns out that there is a more refined notion of depth recalled below, in the sense that it takes the center of the group into account. The definition has already shown in \cite[Definition~3.5]{MR2798422} and \cite{MR2716688}.
\begin{defin}\label{defin:depthcenter}
For any $\alp\in Q_{F}$ and $\gam\in T_{E_{\alp}}(F)^{\rss}$ define the maximal depth $d^{+}(\gam)$ via
$$
d^{+}(\gam) = \max\{ d(z\gam)|\quad z\in Z_{\GL_{2}}(F) \}.
$$
\end{defin}
Note that in \cite[\S 5.3]{MR2716688}, $d^{+}(\gam)$ is denoted by $n(\gam)$.

For any $\alp\in Q_{F}$ a point $x_{\alp}\in \CB(\GL_{2},F)$ can be fixed. Following \cite[p.33]{MR2716688} the decreasing filtration $\{ (\GL_{2})_{x,r}\}_{r\in \BR_{\geq 0}}$ gives a natural filtration on $T_{E_{\alp}}(F)$. Such points are known to exist by \cite{MR2097545}. We will use the specific choices which appear in \cite[\S 5.1]{MR2798422}. For $\eps, 1\in Q_{F}$, set $x_{\eps}=  x_{1}$, both of which are denoted as $x_{0}$. The associated filtrations are defined via $(\GL_{2})_{x_{0},0} = \GL_{2}(\CO_{F})$ and 
$$
(\GL_{2})_{x_{0},r} = 1+\vpi^{\ceil{r}}\RM_{2}(\CO_{F})
$$
for $r>0$. For $\alp\in \{\vpi,\eps\vpi \}$, denote $x_{\alp}$ by $x_{I}$ (where $I$ means "Iwahori"). Then the associated filtrations are defined by
$$
(\GL_{2})_{x_{I},0} 
= \{
\left(
\begin{matrix}
a & b\\
c & d
\end{matrix}
\right)\in \GL_{2}(\CO_{F})
|\quad \ord(b)>0
 \}
$$
and 
$$
(\GL_{2})_{x_{I},r} = 
\left(
\begin{matrix}
1+\Fp^{\ceil{r}}_{F} & \Fp^{\ceil{r+\frac{1}{2}}}_{F}\\
\Fp^{\ceil{r-\frac{1}{2}}}_{F} & 1+\Fp^{\ceil{r}}_{F}
\end{matrix}
\right)
$$
for $r>0$.

We also recall the notion of depth for representations. From \cite{MR1371680}, for any admissible representation $(\pi,V_{\pi})$ of $\GL_{2}(F)$ there exists a unique non-negative real number $d(\pi)\in \BR_{\geq 0}$ that is minimal with the property that there exists $x\in \CB(G,F)$ such that $V^{(\GL_{2})_{x,d(\pi)^{+}}(F)}_{\pi}\neq \{ 0 \}$. The number $d(\pi)$ is called the depth of $\pi$. For an $F$-torus $T$ and $\theta$ a quasi-character of $T(F)$ the concept of depth for $\theta$ can also be introduced. We recall that $d(\theta)\in \BR_{\geq 0}$ is the smallest number for which $\theta|_{T_{d(\theta)^{+}}(F)}$ is trivial but $\theta|_{T_{d(\theta)}(F)}$ is non-trivial. Using the isomorphism $T_{E_{\alp}}\simeq E_{\alp}^{\times}$ the notion of depth for any quasi-characters of $E^{\times}_{\alp}$ can also be introduced.

Now we briefly recall the construction of supercuspidal representations of $\GL_{2}(F)$ which first appeared in \cite{MR492087}. For a tamely ramified quadratic extension $E/F$, a quasi-character $\theta$ of $E^{\times}$ is called \emph{admissible} if the following properties hold.
\begin{itemize}
\item $\theta$ does not factor through $\RN_{E/F}$;

\item If $\theta|_{1+\Fp_{E}}$ factors though $\RN_{E/F}$ then $E$ is unramified.
\end{itemize}
It is shown in \cite{MR492087} that for an admissible pair $(E/F,\theta)$ one can assign a supercuspidal representation $\pi_{\theta}$ of $\GL_{2}(F)$. From \cite{gl2llc} the supercuspidal representation $\pi_{\theta}$ has Langlands parameter $\Ind^{W_{F}}_{W_{E}}\theta$.

We establish the following lemma which is not hard to derive. Note that in the following we always assume that $n$ is coprime to $(q-1)q(q+1)$.

\begin{lem}\label{lem:stilladmissible}
For any admissible pair $(E/F,\theta)$, $(E/F,\theta^{n})$ is also an admissible pair and  $d(\theta) = d(\theta^{n})$.
\end{lem}
\begin{proof}
We first treat the case when $E/F$ is unramified. Then $(E/F,\theta)$ is admissible if and only if $\theta^{\lam}\neq \theta$ for the nontrivial involution $\lam\in \Gal(E/F)$. It turns out that we only need to show $(\theta^{n})^{\lam}\neq \theta^{n}$.
To do so, we establish the following fact.
\begin{num}
\item $\theta$ is not admissible if and only if $\theta|_{\ker\RN_{E/F}}$ is trivial.
\end{num}
The only if part is immediate. We show the if part. If $\theta|_{\ker \RN_{E/F}}$ is trivial, then from the short exact sequence
$$
1\to \ker \RN_{E/F}\to E^{\times}\to \RN_{E/F}(E^{\times})\to 1, 
$$
we get the dual short exact sequence for their dual groups
$$
1\to \wh{\RN_{E/F}(E^{\times})}
\to \wh{E^{\times}}\to \wh{\ker \RN_{E/F}}\to 1.
$$
Since $\theta|_{\ker\RN_{E/F}}$ is trivial, it comes from some quasi-character $\phi$ of $\RN_{E/F}(E^{\times})$, i.e. $\theta = \phi\circ \RN_{E/F}$. But from the short exact sequence
$$
1\to \RN_{E/F}(E^{\times})\to F^{\times}\to \{\pm 1 \}\to 1,
$$
we have the following short exact sequence for their dual groups
$$
1\to \{ \pm 1\}\to \wh{F^{\times}}\to \wh{\RN_{E/F}(E^{\times})}\to 1.
$$
In particular any quasi-character of $\RN_{E/F}(E^{\times})$ lifts to a quasi-character of $F^{\times}$. Therefore we may assume that $\phi$ is a  quasi-character of $F^{\times}$. Hence $\theta$ is not admissible.

It turns out that we only need to show that whenever $\theta|_{\ker \RN_{E/F}}\neq 1$, we also have $\theta^{n}|_{\ker \RN_{E/F}}\neq 1$. Using the identification $E^{\times} = \mu_{E}\times \vpi^{\BZ}\times (1+\Fp_{E})$, we have $\ker \RN_{E/F}\subset \mu_{E}\times (1+\Fp_{E})$. But we have already assumed that $n$ is coprime to $(q-1)q(q+1)$, therefore raising to the $n$-th power gives an automorphism of the dual group of $\mu_{E}\times (1+\Fp_{E})$. Hence $\theta^{n}|_{\ker\RN_{E/F}}\neq 1$ and therefore the quasi-character $\theta^{n}$ is also admissible.

Then we treat the case when $E/F$ is ramified. By definition $\theta$ is admissible if and only if $\theta|_{1+\Fp_{E}}\neq \theta^{\lam}|_{1+\Fp_{E}}$. Again using the fact that $n$ is coprime to $(q-1)q(q+1)$, taking $n$-th power gives an automorphism of the dual group of $1+\Fp_{E}$. Therefore $\theta^{n}|_{1+\Fp_{E}}\neq (\theta^{n})^{\lam}|_{1+\Fp_{E}}$. It follows that the quasi-character $\theta^{n}$ is also admissible.

Finally we are going to establish the equality $d(\theta) = d(\theta^{n})$.

When $d(\theta) = 0$, $\theta$ is the inflation of a nontrivial character of the finite field $(\CO_{E}/\Fp_{E})^{\times}$, which, since $n$ is coprime to $(q-1)q(q+1)$, remains nontrivial after taking $n$-th power. Therefore $d(\theta^{n}) = 0$ as well.

When $d(\theta)>0$, raising to $n$-th power gives an automorphism of the dual group of $(1+\Fp_{E})$. Therefore $d(\theta^{n}) = d(\theta)$.

It follows that we have established the lemma.
\end{proof}
In particular $\pi_{\theta^{n}}$ is a supercuspidal representation of $\GL_{2}(F)$ with $d(\theta) = d(\theta^{n})$.

It turns out that there is a more general construction of tame supercuspidal representations for reductive $p$-adic groups \cite{MR1824988}. \cite[\S 3.5]{MR2431732} shows that the more general constructions in \cite{MR1824988} agree with the earlier work in \cite{MR492087}. From \cite[Remark~3.6]{MR1824988} the equality $d(\pi_{\theta}) = d(\theta)$ holds for any admissible pair $(E/F,\theta)$.  In general, for supercuspidal representation $\pi_{\theta}$ of positive depth $d(\pi_{\theta}) =r>0$, following \cite[p.35]{MR2716688} $\pi_{\theta}$ is compactly induced and is of the form 
$$
\pi_{\theta} = 
c\text{-}\Ind^{\GL_{2}(F)}_{T_{E_{\alp}}(\GL_{2})_{x_{\alp},r/2}}\rho_{\theta}
$$
where $\rho_{\theta}$ is an irreducible representation of $T_{E_{\alp}}(\GL_{2})_{\alp,r/2}$ defined as follows.
\begin{itemize}
\item
If $\Fg\Fl_{x_{\alp},r/2} = \Fg\Fl_{x_{\alp},(r/2)^{+}}$, then 
$$
\rho_{\theta} = \theta.
$$
From \cite[Lemma~3.2.1]{MR2716688}, $T_{E_{\alp}}$ normalizes $(\GL_{2})_{x_{\alp},t}$ for any appropriate $t$, hence the representation $\rho_{\theta}$ is well-defined.

\item
If 
$\Fg\Fl_{x_{\alp},r/2} \neq \Fg\Fl_{x_{\alp},(r/2)^{+}}$, let 
$r_{\theta}$ be the only irreducible component that occurs $\dim(r_{\theta})$ times inside $\Ind^{Z_{\GL_{2}}(F)T_{E_{\alp},0^{+}}(\GL_{2})_{x_{\alp},(r/2)}}_{Z_{\GL_{2}}(F)T_{E_{\alp},0^{+}}(\GL_{2})_{x_{\alp},(r/2)^{+}}}\theta$. The representation $r_{\theta}$ can be extended to an irreducible representation $\rho_{\theta}$ of $T_{E_{\alp}}(\GL_{2})_{x_{\alp},r/2}$.
\end{itemize}
Here we note that the torus $T_{E_{\alp}}$, which normalizes the group $(\GL_{2})_{\alp,r/2}$, might not be exactly the same as the explicit presentation introduced before. For more details one may consult \cite[p.34]{MR2716688}.

For those representations $\pi_{\theta}$ of depth zero, they are compact induction from the inflation to $\GL_{2}(\CO_{F})$ of cuspidal Deligne-Lusztig representations of $\GL_{2}(\CO_{F}/\Fp_{F})$ \cite[p.34]{MR2716688}. 

We establish the following lemma, which will be used in subsequent.

\begin{lem}\label{lem:formaldegreeinv}
With the above notations, $\deg(\pi_{\theta}) = \deg(\pi_{\theta})$, $\dim (\rho_{\theta^{n}}) = \dim (\rho_{\theta})$.
\end{lem}
\begin{proof}
From \cite[p.29]{MR1028263} the following identity holds
$$
\deg(\pi_{\theta}) = \deg(\rho_{\theta})/\vol(Z_{\GL_{2}(F)}(\GL_{2})_{x_{\alp},r/2}/Z_{\GL_{2}(F)})
$$
where $\alp\in Q_{F}$ satisfies $E_{\alp} = E$ and $r= d(\theta)$. Since $d(\theta) = d(\theta^{n})$ and $\deg(\rho_{\theta}) = \dim(\rho_{\theta})$ with suitably chosen measure, we only need to show the identity $\dim (\rho_{\theta}) = \dim(\rho_{\theta^{n}})$. But from \cite[p.35]{MR2716688} the following identity holds
$$
\dim(\rho_{\theta})^{2} = 
[Z_{\GL_{2}}(F)T_{E,0^{+}}(\GL_{2})_{x,r/2}:Z_{\GL_{2}}(F)T_{E,0^{+}}(\GL_{2})_{x,(r/2)^{+}}].
$$
Therefore $\dim(\rho_{\theta})$ also depends only on the depth $r = d(\theta)$. It follows that $\dim(\rho_{\theta^{n}}) = \dim (\rho_{\theta})$ and we have completed the proof of the lemma.
\end{proof}

Now we are ready to state our main result, which is parallel to the Archimedean case in Section \ref{sub:discreteseries}.

\begin{thm}\label{thm:supercuspidal}
Assume that $n$ and $(q-1)q(q+1)$ are coprime. Then the following distributional identity holds,
$$
D_{\GL_{2}}(\gam)^{1/2}\tr_{\pi_{\theta^{n}}}(\gam) = D_{\GL_{2}}(\gam^{n})^{1/2}\tr_{\pi_{\theta}}(\gam^{n}), \quad \gam\in \GL_{2}(F)^{\rss}.
$$
\end{thm}

From the character tables in \cite[Theorem~5.3.2]{MR2716688} and \cite[Theorem~5.4.1]{MR2716688}, for an admissible pair $(E_{\alp}/F,\theta)$ with $\alp \in Q_{F}\bs \{ 1\}$, the distribution character of the supercuspidal representation $\pi_{\theta}$ is supported on the $\GL_{2}(F)$-conjugacy classes of $(\GL_{2})_{x_{\alp},0}$. Therefore we only need to establish Theorem \ref{thm:supercuspidal} for those elements $\gam\in (\GL_{2})_{x_{\alp},0}$ that are regular semi-simple.

We note the following fact.

\begin{lem}\label{lem:discrminant}
For $\gam\in (\GL_{2})_{x_{\alp},0}^{\rss}$, the following equality holds, 
$$
d^{+}(\gam) = d^{+}(\gam^{n}),\quad 
D_{\GL_{2}}(\gam) = D_{\GL_{2}}(\gam^{n}).
$$
\end{lem}
\begin{proof}
We first treat the case that $\gam$ is diagonal. Up to the action of $Z_{\GL_{2}}(F)$ we may assume that $\gam = \diag(1,a)$ for some $a\in \CO^{\times}_{F}$. Note that from \cite[Lemma~3.2.1]{MR2716688} the natural filtration on the diagonal split torus coincides with the filtration induced from $\{(\GL_{2})_{x_{0},r}\}_{r\geq 0}.$ 

By direct computation,
$$
D^{\GL_{2}}(\gam) = |(a-1)(a^{-1}-1)| = |(a-1)|^{2} =q^{-2\ord(a-1)}.
$$
On the other hand, by definition $d^{+}(\gam) = \ord(a-1)$. Therefore 
$$
D^{\GL_{2}}(\gam) = q^{-2d^{+}(\gam)}.
$$

We note that if $a\in \CO^{\times}_{F}\bs 1+\Fp_{F}$, then $a^{n}\in \CO^{\times}_{F}\bs 1+\Fp_{F}$ since $n$ is coprime to $(q-1)q(q+1)$. Therefore $|a-1| = |a^{n}-1|$. On the other hand, if $a\in 1+\Fp_{F}^{k}\bs 1+\Fp_{F}^{k+1}$ for some $k\in \BN$, then $a^{n}\in 1+\Fp_{F}^{k}\bs 1+\Fp^{k+1}_{F}$ as well since $n$ is coprime to $(q-1)q(q+1)$. It follows that $D^{\GL_{2}}(\gam^{n}) = D^{\GL_{2}}(\gam)$.

For $\gam$ non-split, we note that both $D^{\GL_{2}}(\cdot)$ and $d^{+}(\cdot)$ are invariant under (quadratic) field extension \cite[Lemma~3.7]{MR2431235}. It follows that we can reduce the question to the situation that $\gam$ is split since for $\gam$ non-split it always splits over a quadratic extension. In particular the identities $D^{\GL_{2}}(\gam)  =q^{-2d^{+}(\gam)}$ and $d^{+}(\gam^{n}) = d^{+}(\gam)$ holds identically since $n$ is coprime to $(q-1)q(q+1)$. Here we note that the residue field of a quadratic extension of $E$ might have cardinality $q^{2}-1 = (q-1)(q+1)$.

It follows that we have established the lemma.
\end{proof}

Therefore we only need to establish the relation between $\tr_{\pi_{\theta}}$ and $\tr_{\pi_{\theta^{n}}}.$

We divide the proof into two cases, depending on the depth of the supercuspidal representations.

\subsubsection{Depth zero case}
We prove Theorem \ref{thm:supercuspidal} for $\pi$ of depth zero. 

We first recall the following theorem proved by C. Rader and A. Silberger \cite{MR1169888}, which generalized the result of Harish-Chandra \cite{MR0414797}.

\begin{thm}\label{thm:hcradersilberger}
Let $G$ be a connected reductive algebraic group defined over a $p$-adic field $F$ with split center $A_{G}$. 
For any square-integrable representation $\pi$ of $G(F)$, 
$$
\tr_{\pi}(\gam) = 
\frac{\deg(\pi)}{\phi(1)}
\int_{G(F)/A_{G}(F)}
dg^{*}\int_{K}dk \phi({}^{gk}\gam),\quad \gam\in G(F)^{\rss}.
$$
Here $\phi$ is a nonzero $K$-finite matrix coefficient of $\pi$, $dg^{*}$ is a $G(F)$-invariant measure on $G(F)/A_{G}(F)$, $dk$ is the normalized Haar measure of a compact open subgroup $K$ in $G$, and ${}^{gk}\gam = gk\gam(gk)^{-1}$.
\end{thm}
It turns out that there is an explicit construction of matrix coefficients for any supercuspidal representations of $\GL_{2}(F)$ related to the construction of supercuspidal representations as shown in \cite[Theorem~1.9]{MR1039842}.

Following the notation from the explicit construction of supercuspidal representations mentioned above, let $\chi_{\rho_{\theta}}$ be the character of $\rho_{\theta}$. From \cite[Theorem~1.9]{MR1039842} the following function defined on $\GL_{2}(F)$, 
$$
\dot{\chi_{\rho_{\theta}}}(\gam) = 
\bigg\{
\begin{matrix}
\chi_{\rho_{\theta}}(\gam),  & \gam\in T_{E_{\alp}}(\GL_{2})_{x,r/2},\\
0, & \text{ otherwise}.
\end{matrix}
$$ 
is a matrix coefficient for $\pi_{\theta}$. This is exactly the matrix coefficient that we are going to use in Theorem \ref{thm:hcradersilberger}.

Now we plug the matrix coefficient $\dot{\chi_{\rho_{\theta}}}$ into Theorem \ref{thm:hcradersilberger} and we obtain that the distribution character is of the following form 
\begin{equation}\label{eq:characterformula}
\tr_{\pi_{\theta}}(\gam) = 
\frac{\deg(\pi_{\theta})}{\dim(\rho_{\theta})}
\int_{G(F)/A_{G}(F)}
dg^{*}
\int_{K}dk \dot{\chi_{\theta}}({}^{gk} \gam), \quad \gam\in \GL_{2}(F)^{\rss}.
\end{equation}

Combining with Lemma \ref{lem:formaldegreeinv}, we only need to establish the following lemma.

\begin{lem}
For a supercuspidal representation $\pi_{\theta}$ of depth zero, $\dot{\chi}_{\theta}(\gam^{n}) = \dot{\chi_{\theta^{n}}}(\gam)$ for any $\gam\in T_{E}(\GL_{2})_{x,r/2} = T_{E}(\GL_{2})_{x,0}$ that is regular semi-simple.
\end{lem}
\begin{proof}
When $\pi_{\theta}$ has depth zero, the character $\chi_{\theta}$ is the inflation of the Deligne-Lusztig characters of a cuspidal representation of $\GL_{2}(\CO_{F}/\Fp_{F})$ and $E$ is necessarily unramified, i.e. $E= E_{\eps}$ \cite[p.34]{MR2716688}. From \cite{MR1266626}, on the regular semi-simple locus of $\GL_{2}(\CO_{F}/\Fp_{F})$, the character $\chi_{\theta}$ is only supported on the conjugacy classes of elliptic torus that is associated to $E$, and vanishes otherwise. More precisely, the following character formula for $\chi_{\theta}(\gam)$ holds,
$$
\chi_{\theta}(\gam) = 
\bigg\{
\begin{matrix}
\sum_{w\in W(\GL_{n},T_{E})}
\theta(w\cdot\wb{\gam}), & \gam\in T_{E,0}(\GL_{2})_{x_{0},0^{+}} \text{ up to $(\GL_{2})_{x_{0},0}$-conjugacy},\\
0, & \text{ otherwise}.
\end{matrix}
$$
Here $\wb{\gam}$ is the reduction of $\gam$ modulo $(\GL_{2})_{x_{0},+}$. Using the fact that $T_{E,0}$ normalizes $(\GL_{2})_{x_{0},0^{+}}$ \cite[Lemma~3.2.1]{MR2716688}, the identity $\chi_{\theta^{n}}(\gam) = \chi_{\theta}(\gam^{n})$ holds.
\end{proof}

It follows that we have established Theorem \ref{thm:supercuspidal} for supercuspidal representations of depth zero.

\subsubsection{Positive depth case}
We prove Theorem \ref{thm:supercuspidal} for $\pi$ of positive depth.

We only need to establish the following proposition.

\begin{pro}\label{pro:supercuspidalpositivedepth}
Let $\pi_{\theta}$ be a supercuspidal representation of positive depth, then the following distributional identity holds,
$$
\tr_{\pi_{\theta^{n}}}
(\gam) = \tr_{\pi_{\theta}}(\gam^{n}),
\quad \gam\in (\GL_{2})_{x_{\alp},0}^{\rss}.
$$
\end{pro}
\begin{proof}
We will use the character table for positive depth supercuspidal representations from \cite[Theorem~5.3.2, Lemma~5.3.8]{MR2716688}. From Lemma \ref{lem:formaldegreeinv}, we only need to show the following distributional identity
\begin{equation}\label{eq:debackeridentity}
\frac{
\tr_{\pi_{\theta^{n}}}(\gam)}{\deg(\pi_{\theta^{n}})} = 
\frac{\tr_{\pi_{\theta}}(\gam^{n})}{\deg(\pi_{\theta})},
\quad \gam\in (\GL_{2})_{x_{\alp},0}^{\rss}.
\end{equation}

To prove (\ref{eq:debackeridentity}), we divide the proof to the situations depending on the depth of $\gam\in (\GL_{2})_{x_{\alp},0}$. We recall that from Lemma \ref{lem:discrminant} we have $d^{+}(\gam^{n}) = d^{+}(\gam)$. Note that in \cite{MR2716688}, the depth $d^{+}(\gam)$ is denoted as $n(\gam)$. 

When $d^{+}(\gam) =n(\gam)= 0$, we look at the first row of the table showing in \cite[Lemma~5.3.8]{MR2716688}. Since $d^{+}(\gam^{n}) = d^{+}(\gam)$, and $d(\pi_{\theta^{n}}) = d(\theta^{n}) = d(\theta) = d(\pi_{\theta})$, combining with Lemma \ref{lem:formaldegreeinv}, we only need to show the equality $\lam(\rho_{\theta}) = \lam(\rho_{\theta^{n}})$ following the notation in \cite[Lemma~5.3.1]{MR2716688}. But this follows directly from \cite[Lemma~5.3.1]{MR2716688} since $\lam(\rho_{\theta})$ depends only on $d(\theta)$. It follows that the equation (\ref{eq:debackeridentity}) holds when $d^{+}(\gam)  =n(\gam) = 0$.

When $d^{+}(\gam) = n(\gam)>0$, we look at the second and third rows of the table showing in \cite[Theorem~5.3.2]{MR2716688}. Following the notation in \cite[Theorem~5.3.2]{MR2716688} which we will explain in subsequent, we only need to establish the following facts,
\begin{enumerate}
\item\label{(a)} $|\eta(X_{\pi_{\theta}})| = |\eta(X_{\pi_{\theta^{n}}})|$, where $\eta$ is the Weyl discriminant introduced in \cite[\S 4.1]{MR2716688} and $X_{\pi_{\theta}}$ is introduced in \cite[p.35]{MR2716688};

\item\label{(b)} $\gam(X_{\pi_{\theta^{n}}}, {}^wY) = \gam(X_{\pi^{\theta}},{}^w[(1+Y)^{n}-1])$ for any $w\in W$ and $Y\in \Fg^{\p}_{n(\gam)}\bs (\Fz_{n(\gam)}+\Fg^{\p}_{n(\gam)^{+}})$, here $0<n(\gam)<\frac{r}{2}$. The constant $\gam$ is introduced in \cite[p.55]{MR2716688} and $\Fg^{\p}$ is the Lie algebra of $G^{\p} = T_{E_{\alp}}$ appearing in the construction of $\pi_{\theta}$;

\item\label{(c)} $\wh{\mu}_{X_{\pi_{\theta^{n}}}}(Y) = \wh{\mu}_{\pi_{\theta}}([1+Y]^{n}-1)$ for $Y\in V_{n(\gam)}$ and $n(\gam)>\frac{r}{2}$. Here $V_{n(\gam)}$ is introduced in \cite[Definition~1.6.5]{MR2716688};
\end{enumerate}

We first establish the fact $|\eta(X_{\pi_{\theta}})| = |\eta(X_{\pi_{\theta^{n}}})|$. From the beginning of \cite[p.35]{MR2716688}, by definition the following identity holds
$$
X_{\pi_{\theta^{n}}} = 
nX_{\pi_{\theta}}.
$$
In particular, the depth of $X_{\pi_{\theta}}$ and $X_{\pi_{\theta^{n}}}$ are the same. Following Lemma \ref{lem:discrminant} and \cite[Lemma~4.2.1]{MR2716688}, we immediately get (\ref{(a)}).

For (\ref{(b)}), by (\ref{(a)}) we are reduced to show the following identity
$$
\gam(nX_{\pi_{\theta}},{}^w Y) = 
\gam(X_{\pi_{\theta}}, {}^w(nY+f(Y)))
$$
where $f(Y)$ is a polynomial in $Y$ with lowest degree given by $Y^{2}$ lying in $\Fg^{\p}_{2n(\gam)}$. From the definition of $\gam(X,{}^wY)$ on \cite[p.55]{MR2716688}, we are reduced to show the following identity
$$
\gam(X_{\pi_{\theta}}, {}^wY)  =
\gam(X_{\pi_{\theta}},{}^w(Y+f(Y))).
$$
From \cite[p.34-35]{MR2716688}, $X_{\pi}\in \Fg^{\p}_{-r}\bs \Fg^{\p}_{(-r)^{+}}$, and $Y\in \Fg^{\p}_{n(\gam)}\bs (\Fz_{n(\gam)}+ \Fg^{\p}_{n(\gam)^{+}})$ with $0<n(\gam)<r/2$. While in the integral definition of $\gam(X_{\pi_{\theta}}, {}^wY)$ appearing in the bottom of \cite[p.55]{MR2716688}, the element $V\in \Fg_{s^{+}}$ with $s=(r-n(\gam))/2$, therefore we only need to show the following result.

\begin{num}
\item\label{hardlem}
the term appearing in \cite[p.55]{MR2716688}
$$
Q(V,V) = 
\Lam(\tr([X,V]\cdot [V,{}^wY])/2)
$$
is equal to 
$$
\Lam(\tr([X,V]\cdot [V,{}^w(Y+f(Y)])/2).
$$
Here $\Lam$ is a nontrivial additive character of the local field $F$ which has conductor $\Fp_{F}$ \cite[\S 1.1]{MR2716688}.
\end{num}
We establish (\ref{hardlem}). Note that $[X,V]\in \Fg_{(s-r)^{+}}$ and $[V,{}^wf(Y)]\in \Fg_{s+2n(\gam)}$ (here we note that the Weyl group has representatives in $(\GL_{2})_{x_{\alp},0}$, therefore its action on $f(Y)$ preserves the depth of $f(Y)$), from which we deduce $[X,V]\cdot [V,{}^wf(Y)] \in \Fg_{(s-r)^{+}+s+2n(\gam)}\subset \Fg_{n(\gam)}$. In particular $\tr([X,V]\cdot [V,{}^wf(Y)])$ lies in $\Fp_{F}$.
Hence we have established (\ref{hardlem}) and therefore (\ref{(b)}).

For (\ref{(c)}), in general the following Shalika germ expansion holds for $\wh{\mu}_{X}(Y)$,
$$
\wh{\mu}_{X}(Y) = 
\sum_{\CO\in \mathrm{Nil}(\Fg\Fl_{2}(F))}
\Gam_{\CO}(X)\wh{\mu}_{\CO}(Y).
$$
Note that when $X = X_{\pi}$ is elliptic, from \cite[p.171]{MR745601} and \cite[p.181]{MR745602}, the Shalika germs $\{\Gam_{\CO}(X)\}_{\CO\in \mathrm{Nil}(\Fg\Fl_{2}(F))}$ depend only on the depth of $X = X_{\pi}$. In particular $\Gam_{\CO}(X_{\pi})  = \Gam_{\CO}(nX_{\pi})$ for any $\CO\in \mathrm{Nil}(\Fg\Fl_{2}(F))$. 
For the term $\wh{\mu}_{\CO}(Y)$, following \cite[Lemma~5.1]{MR1375619},  $\wh{\mu}_{\CO}(Y)$ also only depends on the discriminant of $Y$ (on group $\GL_{2}(F)$ and Levi subgroups associated to $Y$), which, from Lemma \ref{lem:discrminant} and \cite[Proposition~1.3.1~(c)]{MR2716688}, still depends only on the depth of $Y$. Therefore
$$
\wh{\mu}_{\CO}(Y) = 
\wh{\mu}_{\CO}([1+Y]^{n}-1)
$$
for any $\CO\in \mathrm{Nil}(\Fg\Fl_{2}(F))$ and $Y\in V_{n(\gam)}$. Hence we have established the property (\ref{(c)}). 

It follows that we have established the proposition.

\end{proof}

\bibliographystyle{plain}
\bibliography{SymTransfer_Daniel_Zhilin}

\end{document}